\documentclass[11pt]{amsart}

\usepackage{amsmath,amsthm,indentfirst}
\usepackage{amssymb}
\usepackage{amsfonts}
\usepackage{amscd}
\usepackage[latin1]{inputenc}
\usepackage{hyperref}
\usepackage{ifthen, amsfonts, amssymb, graphicx, srcltx, mathrsfs, xfrac,
amsmath}
\usepackage{enumerate}
\usepackage{color}
\usepackage{tikz}
\usepackage{tikz-cd}

\theoremstyle{plain}
\newtheorem*{maintheorem*}{Main Theorem}

\newtheorem*{remark*}{Remark}
\newtheorem*{conjecture*}{Conjecture}
\newtheorem*{prop*}{Proposition}
\newtheorem{thm}{Theorem}[section]

\newtheorem{lem}[thm]{Lemma}
\newtheorem{prop}[thm]{Proposition}

\theoremstyle{definition}

\newtheorem*{proofc*}{Proof of Theorem C}

\newtheorem{conjecture}[thm]{Conjecture}

\newtheorem{definition}[thm]{Definition}

\newtheorem{remark}[thm]{Remark}

\newtheorem{claim}[thm]{Claim}

\newcommand\restate[3]{\medskip\par\noindent
  {\bf #1 \ref{#2}.} {\em #3}\medskip}

\DeclareMathOperator{\CAT}{CAT}
\DeclareMathOperator{\Teich}{Teich}
\DeclareMathOperator{\Mod}{Mod}

\DeclareMathOperator{\WP}{WP}

\DeclareMathOperator{\tw}{tw}

\DeclareMathOperator{\diam}{diam}

\DeclareMathOperator{\na}{na}

\DeclareMathOperator{\dist}{dist}
\DeclareMathOperator{\ML}{\mathcal{ML}}
\DeclareMathOperator{\GL}{\mathcal{GL}}
\DeclareMathOperator{\EL}{\mathcal{EL}}
\DeclareMathOperator{\dw}{d_{wp}}

\DeclareMathOperator{\calC}{\mathcal{C}}

\DeclareMathOperator{\calI}{\mathcal{I}}
\DeclareMathOperator{\calL}{\mathcal{L}}
\DeclareMathOperator{\calM}{\mathcal{M}}
\DeclareMathOperator{\calN}{\mathcal{N}}
\DeclareMathOperator{\calP}{\mathcal{P}}

\DeclareMathOperator{\calS}{\mathcal{S}}

\DeclareMathOperator{\calY}{\mathcal{Y}}
\DeclareMathOperator{\calX}{\mathcal{X}}
\DeclareMathOperator{\calZ}{\mathcal{Z}}

\DeclareMathOperator{\Mark}{Mark}

\DeclareMathOperator{\CC}{\mathbb{C}}

\DeclareMathOperator{\HH}{\mathbb{H}}
\DeclareMathOperator{\NN}{\mathbb{N}}
\DeclareMathOperator{\RR}{\mathbb{R}}
\DeclareMathOperator{\ZZ}{\mathbb{Z}}

\newcommand{\T}{Teichm\"{u}ller }
\newcommand{\W}{Weil-Petersson }
\newcommand{\ep}{\epsilon}
\newcommand{\bd}{\partial}
\newcommand{\tn}{\textnormal}
\newcommand{\pA}{pseudo-Anosov }
\newcommand{\nbhd}{neighborhood }
\newcommand{\nbhds}{neighborhoods }

\newcommand{\oT}[1]{\overline{\Teich(#1)}}
\newcommand{\ssm}{\smallsetminus}
\newcommand{\union}{\cup}
\newcommand{\intersect}{\cap}
\newcommand{\seg}{\overline}
\newcommand{\oseg}{\overrightarrow}
\newcommand{\half}{\frac{1}{2}}
\def\<#1,#2>{\langle#1,#2\rangle}
\def\<#1>{\langle#1\rangle}

\numberwithin{equation}{section}

\long\def\realfig#1#2{
  \begin{figure}[htbp]
    \begin{center}
\graphicspath{ {./onelipschitz/} }
\includegraphics {#1}
\caption[#1]{#2}
\label{#1}
    \end{center}
\end{figure}}

\begin{document}


\title[Bottlenecks for WP geodesics]{Bottlenecks for Weil-Petersson geodesics}

\date{\today}


\author{Yair Minsky}
\address{Department of Mathematics, Yale University, 10 Hillhouse Ave, New Haven, CT, 06511}
\email{yair.minsky@yale.edu}

\author{Babak Modami}
\address{Institute for Mathematical Sciences, Stony Brook University, Stony Brook, NY, 11794-3660}
\email{babak.modami@stonybrook.edu}

\thanks{Minsky was supported by NSF grant DMS-161087.}


\date{\today}

\begin{abstract}
We introduce a method for constructing Weil-Petersson (WP) geodesics with certain behavior in the Teichm\"{u}ller space.
This allows us to study the itinerary of geodesics among the strata of the WP completion
and its relation to subsurface projection coefficients of their end invariants.
As an application we demonstrate the disparity between short curves in the universal curve
over a WP geodesic and those of the associated hyperbolic $3$--manifold. 
\end{abstract}
\maketitle

\setcounter{tocdepth}{1}

\section{Introduction}

In this paper we explore some questions about visibility in the Weil-Petersson geodesic
flow in Teichm\"uller space, and its connections to {\em synthetic} aspects of the flow, by which we mean the
combinatorial behavior of geodesic flow lines for large times. Our motivation is partly
the analogy between Weil-Petersson flow and Teichm\"uller flow, and partly the connections between Weil-Petersson geometry
and the geometry of hyperbolic $3$--manifolds. 

Our main result is a criterion (Theorem \ref{thick bottleneck}) for existence of {\em bottlenecks} between certain (non-recurrent)
pairs of geodesics, and its consequences for visibility, that is to say connectivity by
geodesics of certain points at infinity. We use this result
to construct examples of geodesics whose approach pattern to the completion strata of the
\T space exhibits some new phenomena (Theorems \ref{thm:WP mismatch},
and \ref{thm:one-step filling}),
and related
examples in which the connection
between WP geometry and hyperbolic geometry breaks down (Theorems \ref{thm:QF mismatch}
and \ref{thm:fibered mismatch}).

\subsection*{Bottlenecks and visibility}
A {\em bottleneck} for a pair $R,R'$ of subsets of a geodesic metric space $\calX$ is a 
compact set $K\subsetneq \calX$ such that every geodesic segment with endpoints on
$R$ and $R'$ meets $K$. 

For two geodesic rays $r,r'$, {\em visibility} of their endpoints at
infinity is the existence of a geodesic $g$ strongly asymptotic to $r$
in forward time and $r'$ in backward time. The existence of a
bottleneck is an important step in the proof of visibility. 

For example, it is shown in \cite[Theorem 1.3]{bmm1} that any two {\em
  recurrent} geodesic rays have a bottleneck and satisfy the corresponding
visibility property.  To consider the non-recurrent case we start with geodesics that are
asymptotic to {\em completion strata} in Teichm\"uller space. If $\omega$ is a multicurve
in $S$, let $\ell_\omega:\Teich(S) \to \RR_+$ be its geodesic length
function on Teichm\"uller space. 
The stratum $\calS(\omega)$ is the locus in the Weil-Petersson
completion of $\Teich(S)$ where the curves of $\omega$ are replaced by punctures (hence
$\ell_\omega = 0$). 
 The following conjecture seems reasonable but is currently
beyond our reach: 

\begin{conjecture}\label{bottleneck for strata}
  Let $\omega,\omega'$ be two multicurves in $S$ 
  that fill the surface. Then 
  $\calS(\omega)$ and $\calS(\omega')$ have a bottleneck.
\end{conjecture}

Our main technical result will be the following restricted version:

\begin{thm}\label{thick bottleneck}
  Let $\omega,\omega'$ be two co-large multicurves that fill $S$. Let $\bar\ep>0$, and
let $r$ and $r'$ be infinite length WP geodesic rays in $\oT{S}$
that are strongly asymptotic to (or contained in) the $\bar\ep$--thick parts of the
strata $\calS(\omega)$
and $\calS(\omega')$, respectively.  Then 
$r$ and $r'$ have a bottleneck.
\end{thm}
Here a multicurve is {\em co-large} if it is the boundary of a subsurface all of whose complementary subsurfaces are annuli and
three-holed spheres. See Section \ref{sec:background} for details. 

The corresponding visibility statement is given in 
      Theorem \ref{thm:asymplargevis}.

\subsection*{Itinerary and subsurface projections}
A motivating question for us is finding a combinatorial or symbolic description of the
{\em itinerary} of a Weil-Petersson geodesic, by which we mean the list of completion strata
that the geodesic approaches. The classical analogue is a geodesic in the modular surface
$\HH^2/SL(2,\ZZ)$ whose approaches to the cusp are determined by the continued-fraction
expansion of its endpoint in $\mathbb{R}{\mathrm P}^1 = \bd \HH^2$.

A natural generalization of the continued-fraction coefficients is given by {\em
  subsurface projection coefficients (or just subsurface coefficients)}, as developed in \cite{mm2} and \cite{elc1}. Rafi studied the
relation of these projections to the itineraries of Teichm\"uller geodesics in
\cite{rshteich,rcombteich,rteichhyper}.  A geodesic (in both the Teichm\"uller and Weil-Petersson settings) has
a pair $(\nu^+,\nu^-)$ of endpoints, which can be points in $\Teich(S)$, or laminations 
(with or without transverse measure),
and for any essential subsurface $Y\subseteq S$ we consider
\[
d_Y(\nu^+,\nu^-) := \diam_{\calC(Y)}\Big(\pi_Y(\nu^+),\pi_Y(\nu^-)\Big)
\]
where $\pi_Y$ is the projection to the curve complex $\calC(Y)$ (see Section
\ref{sec:background} for detailed definitions).

Very roughly, when these coefficients are large, the geodesic makes close approaches to
the strata of $\Teich(S)$ (equivalently, $\inf_g \ell_\gamma$ is small
for some $\gamma$), but the complete correspondence is not fully understood.

Rafi showed, for a Teichm\"uller geodesic $g$ with end invariant
$(\nu^+,\nu^-)$, that lower bounds on  
$d_Y(\nu^+,\nu^-)$ imply upper bounds on $\inf_{g} \ell_{\bd Y}$. However he developed
sequences of examples showing that the opposite implication fails. 

Using Theorem \ref{thick bottleneck} we are able to produce examples of WP geodesics for
which the analogue of Rafi's result holds: 
\begin{thm}\label{thm:WP mismatch}
There exist $A\geq 1,\ep_0>0$ so that for any $\ep>0$ there is a WP geodesic segment
$\seg{pq}$ whose endpoints are in the $\ep_0$--thick part of $\Teich(S)$, and a
curve $\gamma$ so that
$$
\inf_{x\in \seg{pq}}\ell_\gamma(x) < \ep
$$
whereas
$$
\sup\Big\{d_Y(p,q) \ | \ Y\subseteq S, \gamma\subseteq\bd Y\Big\} \le A.
$$
\end{thm}

A more detailed description of this construction appears in Theorem \ref{thm:one-step
  filling}. We obtain a phenomenon we
might call {\em indirect shortening}, in which, while a curve $\gamma$ with small $\inf_{x\in g}
\ell_\gamma(x)$  is not in the boundary of any subsurface $Z$ with large $d_Z(\nu^-,\nu^+)$,
it is in the boundary of a subsurface $Z$ which in turn contains enough
subsurfaces $Y$ with large $d_Y(\nu^+,\nu^-)$ to fill it. We discuss this further in
Section \ref{sec:example}.

In Section \ref{sec:nonannularbddcomb}, we consider one case in which the 
correspondence between large subsurface projections and close approaches to strata is simple
and direct: A geodesic has {\em non-annular bounded combinatorics} when there is an upper
bound on all projection coefficients $d_Y(\nu^+,\nu^-)$ except when $Y$ is an annulus.
Theorem \ref{thm:short curve bddcomb} shows that, for a geodesic $g$ satisfying such a
condition, a curve $\alpha$ with small $\inf_g\ell_\alpha$ is the core of an annulus $Y$
with large $d_Y(\nu^+,\nu^-)$, and vice versa. The proof of this mostly assembles existing
techniques, as outlined in Section \ref{sec:background}, and we include it here for
completeness.

\subsection*{Comparison with Kleinian groups} 
Using the methods developed in
\cite{elc1,elc2}, one can convert Theorem \ref{thm:WP mismatch} to a
statement comparing the geometry of WP geodesics and hyperbolic
3-manifolds. For any two points $p,q\in \Teich(S)$ there is a
quasi-Fuchsian representation $\rho:\pi_1(S)\to PSL(2,\CC)$ such that  
$QF(p,q):=\HH^3/\rho(\pi_1(S))$ has conformal boundary
surfaces $p$ and $q$. One can ask, as in \cite{yairbddgeomkl} for Teichm\"uller geodesics,
about the correspondence 
between short geodesic curves in $QF(p,q)$ and curves with short length along the WP geodesic
$\seg{pq}$. The following theorem indicates that the correspondence is not complete:

\begin{thm}\label{thm:QF mismatch}
There exists $\ep_1>0$ so that for any $\ep>0$ there is a pair $(p,q)\in
 \Teich(S)\times \Teich(S)$ and a curve $\gamma$ in
 $S$ such that
 $$
 \inf_{x\in \seg{pq}} \ell_\gamma(x) < \ep
 $$
 whereas
 $$
 \ell_\gamma(QF(p,q)) \geq \ep_1.
 $$
\end{thm}

With some more care one can obtain a similar statement for fibered 3-manifolds and their
associated WP geodesic loops. For a pseudo-Anosov map $\Phi\in \Mod(S)$ let $M_\Phi$ denote the
associated hyperbolic mapping torus and $A_\Phi$ the WP axis of $\Phi$ in $\Teich(S)$.

\begin{thm}\label{thm:fibered mismatch}
There exists $\ep_1>0$ so that for any $\ep>0$ 
there is a pseudo-Anosov $\Phi\in\Mod(S)$ and a curve $\gamma$ in $S$
such that
$$\inf_{x\in A_\Phi}\ell_\gamma(x) <\ep$$
whereas
$$
\ell_\gamma(M_\Phi) > \ep_1.
$$
\end{thm}

We note that similar statements, for comparing Teichm\"uller geodesics and Kleinian
groups, follow from Rafi's results (see discussion in \cite{yairbddgeomkl}), but the
actual set of examples, as well as the proofs, are quite different.

\subsection*{Brief historical sketch}
Several important geometric and dynamical properties of the
Weil-Petersson metric were established over the last decade; see
Wolpert \cite{wol} for a summary of
some these results. Our point of view begins with
\cite{bmm1,bmm2} which introduced ending laminations for WP geodesics
and studied the question of itineraries and their relation to
subsurface projections.

Some of these techniques were developed further in
\cite{wpbehavior,asympdiv},  and recently ending laminations were applied
successfully to determine limit sets of WP geodesics in Thurston's
compactification of \T space  
 \cite{limitsetwp,limitsetwp-nonminimal}
exhibiting various exotic asymptotic behavior of the geodesics; for example geodesics with non simply connected (circle) limit sets.
Moreover, Hamenst\"adt \cite{Ham:wp-teich} used ending laminations to establish certain measure theoretic properties of the WP geodesic flow.
On the other hand, Brock and Modami in \cite{wprecurnue} showed that an analogue of the Masur criterion \cite{masurcriterion} does not hold for 
WP geodesics and the associated ending laminations.

However a complete description of the WP geodesic flow in terms of
ending laminations remains elusive. It is our hope that the examples
and techniques developed here will add to the toolkit for addressing
the issue more fully. 

\subsection{Plan of the paper} In Section \ref{sec:background} we provide some background
and supplementary results about coarse geometry of 
curve complexes and other related complexes, and recall definitions and techniques for handling
the \W metric. In Section
\ref{sec:nonannularbddcomb} we prove Theorem \ref{thm:short curve bddcomb} which completes
the itinerary picture for geodesics satisfying
the non-annular bounded combinatorics
condition (no indirect curve shortening occurs in this situation). In
section \ref{bottlenecks} we prove our main theorem about the existence of bottlenecks for
certain families of geodesic segments (Theorem \ref{thick bottleneck}). In
Section \ref{sec:example} we prove Theorem \ref{thm:one-step filling}, which uses the
Bottleneck theorem to construct WP geodesic segments that have the indirect curve
shortening property. In particular we obtain a proof of Theorem \ref{thm:WP mismatch}.
In Section \ref{sec:closed}  we prove Theorem \ref{thm:closedgeodesic}, which produces
closed WP geodesics that have the indirect curve shortening
property. The delicacy here is to approximate the segments constructed in Theorem
\ref{thm:one-step filling} with arcs of closed geodesics while controlling end invariants
and their subsurface projection coefficients. 
In Section \ref{sec:kleinian} we show how Theorems \ref{thm:one-step filling} and
 \ref{thm:closedgeodesic} translate to 
Theorem \ref{thm:QF mismatch} and Theorem \ref{thm:fibered mismatch}, which 
indicate a mismatch between 
the short curves of  WP geodesics and the short curves of the corresponding
hyperbolic $3$--manifolds.

\section{Background}\label{sec:background}

In this section we set notation and recall a variety of facts from the literature. 
Some results are just quoted from the literature, for some we outline the proofs, and 
a few require a short argument which we supply.

\subsection{Curves and surfaces}\label{subsec:CC} 

Let $S$ be a connected, orientable surface of finite type. 
 In this paper by a {\em curve} $\alpha$ on $S$ we mean the homotopy class of an essential
 (i.e. homotopically nontrivial and nonperipheral -- not homotopic to a puncture or boundary) simple closed curve on $S$, 
and by a {\em subsurface} $Y\subseteq S$ we mean the homotopy class of a closed,
connected, nonperipheral, $\pi_1$-injective subsurface of $S$. 
 A {\em multicurve} on $S$ is a set of pairwise disjoint non-parallel curves on $S$. 
 
We abuse notation a bit to blur the distinction between a subsurface and its interior; for example
if $Y\subseteq S$ is a subsurface we take $\Teich(Y)$ to mean the same thing as $\Teich({\rm
  int}(Y))$. This is convenient when we consider subsurfaces in the complement of multicurves
or other subsurfaces on $S$.
 
 When two curves or multicurves $\alpha,\beta$ cannot be realized disjointly on a surface we say that they {\em overlap} 
 and denote $\alpha\pitchfork\beta$. Similarly, when a curve $\alpha$ and a subsurface $Y$ cannot be realized disjointly, 
 we say that they overlap and denote $\alpha\pitchfork Y$.
We say that two (multi)curves $\alpha,\beta$ {\em fill} the surface $S$ if their union
intersects every curve in $S$; equivalently if, when realized with minimal
intersection number, the complement of $\alpha\union\beta$ is a union of disks and
peripheral annuli.

Thurston's measured lamination space $\ML(S)$ is a natural completion of the set of curves
and multicurves, and we will also consider the space of geodesic laminations (without measures)
$\GL(S)$. (Laminations are geodesic with respect to a reference hyperbolic metric as
usual, but the choice of metric doesn't matter here). See \cite{FLP,Bon:geodcurrent} for basic facts about these spaces.
Within $\GL(S)$ let $\EL(S)$ denote the space of minimal filling laminations: A lamination
is filling if it intersects every simple closed geodesic; equivalently if its
complementary regions are ideal polygons or once-punctured ideal polygons.

The natural weak-$*$ topology on $\ML(S)$ descends to the {\em coarse Hausdorff}
topology on the supports in $\GL(S)$. In particular $\EL(S)$ with the coarse Hausdorff
topology is a Hausdorff space (no pun intended), and convergence is characterized as follows:
$\lambda_n \to\lambda$ in the coarse Hausdorff topology on $\EL(S)$ if any accumulation point of $\{\lambda_n\}_n$ in
the Hausdorff metric on closed subsets of $S$ contains $\lambda$ as a sublamination. See 
\cite{Ham:trbdry} and \cite[$\S 2$]{gabai:endlamspace} for details.

The following class of subsurfaces and multicurves plays a special role throughout the paper: 
 \begin{definition}\label{def:large}
We say that a subsurface $Z\subseteq S$ is {\em large } if each connected component of $S\ssm Z$ is either a three holed sphere or an annulus.
 The boundary of a large subsurface is called a {\em co-large multicurve}.
\end{definition}

\begin{remark} \label{remark:sub co-large}
It is easy to verify that any submulticurve of a co-large multicurve is 
 a co-large multicurve.
\end{remark}

\subsection{Weil-Petersson geometry}\label{subsec:WP geom}
Consider now $S$ with 
negative Euler
characteristic, and let $\Teich(S)$ denote
the {\em \T space} of marked complete finite-area hyperbolic surfaces homeomorphic to
$S$. The {\em mapping class group} of the surface,  $\Mod(S)$, is the group of orientation preserving homeomorphisms of the surface up to isotopy.
The mapping class group acts on the \T space by remarking (precomposition with homeomorphisms) and the quotient is the moduli space of Riemann surfaces $\calM(S)$.

\medskip

The {\em \W (WP) metric} on $\Teich(S)$ is an incomplete Riemannian metric which is
invariant under the action of $\Mod(S)$ and hence descends to a metric on $\calM(S)$. 
We will recall the basic facts about the WP metric that we will need; for a more complete
account see Wolpert's survey \cite{wol}.

The curvature of the metric is strictly negative, but not bounded away from 0 or
$-\infty$. Moreover, the WP metric is geodesically convex \cite[Theorem 3.10]{wol}:  there is a unique geodesic between any two points $x,y\in \Teich(S)$
which we denote by $\seg{xy}$. We typically think of geodesics parameterized by arclength.

We denote the WP distance function as $\dw$, or just $d$ when confusion is unlikely.
The completion of $(\Teich(S),\dw)$, denoted by $\oT{S}$, is a stratified $\CAT(0)$ space where
 each stratum consists of marked surfaces pinched at a multicurve $\sigma$. We denote the
 stratum of the multicurve $\sigma$ by $\calS(\sigma)$, with $\calS(\emptyset) = \Teich(S)$. 
To describe the metric on $\calS(\sigma)$, let surfaces $X_j,\; j=1,\ldots, k$, be the connected components of $S\ssm \sigma$ which are not three-holed spheres, where punctures are introduced on $X_j$ at curves in $\sigma$.
Then $\calS(\sigma)$ is totally geodesic in $\oT S$, and can be identified with
 \[\Teich(S\ssm\sigma)\cong\prod_j \Teich(X_j)\]
  where the completed WP metric on $\calS(\sigma)$ is isometric to the
Riemannian product of WP metrics on $\Teich(X_j)$; see \cite{maswp}.

\subsubsection*{Length-functions}
For a curve or multicurve $\alpha\subseteq\calC(S)$ the length-function
   \[\ell_\alpha\colon \Teich(S)\to\RR_+\]
    assigns to a point $x$ the sum of the lengths of the geodesic representatives of connected components of $\alpha$ at
    $x$.

We also note that $\ell_\alpha$ extends continuously to
$$\ell_\alpha:\oT{S}\to[0,\infty],$$
where
$\{\ell_\alpha = 0\}$ is the closure of $\calS(\alpha)$ and $\{\ell_\alpha=\infty\}$ is
the union  of strata $\calS(\sigma)$  for which $\alpha \pitchfork \sigma$. 

Given $\ep>0$ recall that the {\em $\ep$--thick part} of \T space consists of all points $x\in \Teich(S)$ so that $\ell_\alpha(x)\geq 2\ep$ for all curves $\alpha$.
Its complement is called the {\em $\ep$--thin part}.

  The {\em Bers constant} $L_S>0$ of a surface $S$ with negative Euler characteristic is a
  number depending only on the topological type of the surface so that any $x\in\Teich(S)$
  has a pants decomposition, called a {\em Bers pants decomposition}, with the property
  that the length of all curves in the pants decomposition are at most $L_S$; see \cite[\S
    4.1]{buser}. 
  
  \medskip
  We recall also that Wolpert proved  that the length-functions $\ell_\alpha$ are
  {\em strictly convex} in $\Teich(S)$, that is, for any WP geodesic $g$ the function
  $\ell_\alpha\circ g$ has positive second derivatives
  \cite[Corollary 4.7]{wolpert:nielsen} (see also $\S 3$ of \cite{wol}; in particular
  Theorem 3.9).

\subsubsection*{Thick regions of strata}  
For $\ep>0$ we define the $\ep$--thick part of a stratum, denoted by $\calS_\ep(\sigma)$,
to be the product of the $\ep$--thick parts of its factors $\Teich(X_j)$ 
    where $X_j$ are the connected components of $S\ssm \sigma$. 

For $d>0$ we denote the $d$--neighborhood of $\calS_\ep(\sigma)$ by
   
    \begin{equation}\label{eq:Uregion}
    U_{d,\ep}(\sigma):=\calN_d(\calS_\ep(\sigma)).
    \end{equation}

For sufficiently small neighborhoods of $\calS_\ep(\sigma)$ we retain some control of
length-functions:
\begin{lem}\label{lem : b-nbhd}
For any $\ep>0$ sufficiently small there is $b>0$ so that:
 For any $\beta\notin\omega$ and $x\in U_{b,\ep}(\omega)$, 
 $\ell_\beta(x)$ is uniformly bounded below.

Given $b$ there is $\ep_b>0$ such that, in the $b$--\nbhd of any point in $\oT{S}$, there
is a point with injectivity radius $\ep_b$ outside the cusps.
\end{lem}

\begin{proof}
  Let $\Gamma\subset \Mod(S)$ be the stabilizer of $\omega$, or equivalently
  the stabilizer of $\calS(\omega)$. 
Note that $\calS_\ep(\omega)/\Gamma$ is a compact subset of $\oT{S}/\Gamma$ (it is
   the $\ep$--thick part of the moduli space of $S\ssm\omega$).

Define $f(x) := \inf\{\ell_\beta(x) : \beta\notin\omega\}$ on $\oT{S}$. This is a continuous, $\Gamma$-invariant function
and it is strictly positive on  $\calS_\ep(\omega)$, by definition. It descends to a
continuous function on $\oT{S}/\Gamma$ and, since $\calS_\ep(\omega)/\Gamma$ is compact, there is
some $b>0$ such that it is still strictly positive on the closure of the $b$--neighborhood
of $\calS_\ep(\omega)/\Gamma$. Lifting back to $\oT{S}$ we have the desired first
statement.

The second statement follows
directly from compactness of the completion of the moduli space. 
\end{proof}

\subsection{Coarse geometric models}\label{subsec:coarse-model-CC}
We recall here the system of complexes and their projection maps which can be used to give
rough models for Teichm\"uller space and for the mapping class group. We refer to 
\cite{mm1,mm2} and \cite{br,bkmm} for the details and basic facts about these complexes.

\subsubsection*{Curve complexes}
We denote the curve complex of $S$ by $\calC(S)$, defined so that
$k$-simplices are $(k+1)$-component multicurves (with minor exceptions for one-holed tori,
4-holed spheres and annuli). We may turn the complex to a metric complex by declaring that
each simplex is the Euclidean simplex with side lengths $1$. The seminal result of Masur and Minsky  \cite{mm1}
showed that this metric complex is Gromov hyperbolic.

The {\em pants graph} $\calP(S)$ is the graph whose vertices are pants decompositions,
i.e. maximal simplices of $\calC(S)$, and whose edges are pairs of pants decompositions
related by an elementary move consisting of replacing a curve with another that intersects
it as few times as possible. We can turn the graph to a metric graph by assigning length $1$ to each edge.

A {\em marking} of $S$ is a filling collection of curves consisting of a pants decomposition, called the base of marking,
together with curves transverse to each component, as discussed in \cite{mm2}. The {\em marking graph}
 $\Mark(S)$ is formed by defining elementary moves between markings, in such
a way that $\Mark(S)$ is connected and quasi-isometric to $\Mod(S)$.

\subsubsection*{Subsurface projections}
If $Y$ is a non-annular subsurface of $S$ and $\alpha$ is a curve in $S$ intersecting $Y$, we can
define $\pi_Y(\alpha)$ by taking arcs of intersection of $\alpha$ with $Y$ (once they are
in minimal position) and replacing them by curves using a mild surgery (closing up with subarcs of $\bd Y$). When $Y$ is an
annulus we define $\calC(Y)$ to be the complex of essential arcs in the natural compactifoed annular
cover $\widehat Y$ of $S$ associated to $Y$, and form $\pi_Y(\alpha)$ by lifting $\alpha$ to
this cover (see e.g. \cite[\S 2]{mm2} or \cite[\S 4]{elc1}). One way to handle the arbitrary choices involved in these definitions is to
let $\pi_Y(\alpha)$ denote the set of all possibilities and check that this set has
uniformly bounded diameter. If $\alpha$ does not intersect $Y$ we let
$\pi_Y(\alpha)=\emptyset$.

The definition extends to pants decompositions and markings by taking a union over their
components, and to laminations provided their intersection with $Y$ does not contain
infinite leaves. In particular $\pi_Y(\lambda)$ makes sense if $\lambda\in \EL(S)$.

Finally we extend the definition to $\pi_Y(x)$ where $x\in \Teich(S)$ by letting $\mu(x)$
denote a {\em Bers marking} of $S$, namely a marking whose
base pants decomposition is a Bers pants decomposition 
and whose transversal curves are chosen with minimal lengths. (If there is more than one
such marking we make an arbitrary choice). We then define $\pi_Y(x) := \pi_Y(\mu(x))$.

The {\em $Y$ subsurface coefficient} $d_Y(p,q)$, for any $p,q$ whose projections to $Y$
are defined as above and are nonempty, is now defined by
\begin{equation}\label{eq : subsurf cpeff}
d_Y(p,q):=\diam_{\calC(Y)}\Big(\pi_Y(p)\cup\pi_Y(q)\Big).
\end{equation}
We usually do not distinguish between a annulus $Y$ and its core curve $\alpha$, 
for example denoting $\pi_Y$ by $\pi_\alpha$ and $d_Y$ by $d_\alpha$.
From the definition it is clear that $d_Y$ satisfies the triangle inequality, provided
all three projections are nonempty.

\subsubsection*{Hierarchy paths}
We recall that hierarchy (resolution) paths 
form a transitive set of quasi-geodesics in the pants or marking graph of a surface
with quasi-geodesic constants that depend only on the topological type of the surface.
An important property of hierarchy paths is the no--backtracking property
\cite[\S 4]{mm2}, which 
we state here in a form that will serve our purpose in Section
\ref{sec:nonannularbddcomb}. 
 \begin{prop}\tn{(No-backtracking property)}\label{prop:nobacktr}
 Let $\rho:I\to\calP(S)$ be a hierarchy resolution path, 
 and let $i,j,k,l\in I$ with $i\leq j\leq k\leq l$, then for a non-annular subsurface $Y$ we have that
 \[d_Y(\rho(i),\rho(l))\geq d_Y(\rho(j),\rho(k))-M\]
 \end{prop}
For a morecomplete list of properties of hierarchy paths 
see \cite[\S 4,5]{mm2}, \cite[\S 5]{elc1}, \cite[Theorem 2.6]{bmm2} and \cite[Theorem 2.17]{wpbehavior}.

\subsubsection*{Pants graph and WP metric}  
  Brock \cite[Theorem 3.2]{br} showed that the (coarse) map 
  \begin{equation}\label{eq:brock}
  Q:\Teich(S)\to \calP(S)
  \end{equation}
   that assigns to a point $x\in \Teich(S)$ a Bers pants decomposition is a quasi-isometry with constants that depend only on the topological type of the surface.
   
   Here we also recall the Masur-Minsky distance formula \cite[Theorem 6.12]{mm2} which provides a coarse estimate for the distance of any two pants decompositions $P,Q\in \calP(S)$:
   Given $A>0$ large enough there are $K\geq 1$ and $C\geq 0$ so that
   \begin{equation}\label{eq:dist}
   d_{\calP}(P,Q)\asymp_{K,C} \sum_{Y\subseteq S:\; \na } \{d_Y(P,Q)\}_A
   \end{equation}
   holds, where $\{x\}_A=\begin{cases}x\;\; \text{if}\; x\geq A \\ 0\;\; \text{if}\;
   x<A  \end{cases}$ is the cut-off function. The ``$\na$'' in the above
   formula stands for non-annular 
   and indicates that the sum is over non-annular subsurfaces.

   \medskip

Brock's quasi-isometry (\ref{eq:brock}) combined with the  distance formula (\ref{eq:dist}) gives us a coarse formula for the \W distance:
\begin{equation}\label{eq:distWP}
 \dw(x,y)\asymp_{K,C} \sum_{Y\subseteq S:\; \na} \{d_Y(x,y)\}_A 
 \end{equation}
where $A\geq 1$ is large enough and $K,C$ depend on $A$.

The following immediate consequence of the distance formula can also be obtained by more
elementary means: 
\begin{lem}\label{lem:dWP-dY}
For any $a>0$, there is a $D\geq 1$, so that if $\dw(x,y)\leq a$ then 
\[\sup_{\substack{Y\subseteq S:\; \na}}d_Y(x,y)\leq D\]
\end{lem}

We also need the following lemma which gives bounds on subsurface projections for
convergent sequences of laminations: 

\begin{lem}\label{lem:dY1}
Let $\lambda$ be a lamination in $\EL(S)$ and $\gamma$ a curve on $S$.
Then, there is a neighborhood $U$ of $\lambda$ in the coarse Hausdorff topology
on $\EL(S)$ such that for all laminations $\mu$ in $U$ we have
 \[\sup_{\substack{Y\subseteq S: \gamma\subseteq \bd Y}} d_Y(\lambda,\mu)\leq 4.\]
\end{lem}
\begin{proof}
Equip the surface $S$ with a complete hyperbolic metric and realize $\gamma$ and $\lambda$
geodesically. Let $l$ be a leaf of $\lambda$ that intersects $\gamma$, and let $a$ be a subarc of $l$
with end points on $\gamma$ that is essential in the subsurface $S\ssm \gamma$.
When $\gamma$ is a separating curve choose an arc as above in each connected component of $S\ssm \gamma$.

Let $R_a$ denote a small regular neighborhood of $a$ in $S\ssm\gamma$
which is of the form $a\times[0,1]$ where $(\bd a)\times[0,1]$ is two arcs on $\gamma$. 

A sequence of laminations $\mu_i\in \EL(S)$ converges in the coarse Hausdorff topology to
$\lambda$ if any accumulation point of $\mu_i$ in the Hausdorff metric is a lamination
containing $\lambda$, and in particular the leaf $l$. Thus, we can choose a neighborhood
$U$ of $\lambda$ in the coarse Hausdorff topology such that any $\mu\in U$ (realized
geodesically) has a leaf $l'$ passing through $R_a$ 
from one side on $\gamma$ to the other side on $\gamma$. 
Denote the subarc of $l'$ with end points on $\gamma$ by $a'$.

Now let $Y\subseteq S$ be a non-annular subsurface with geodesic boundary $\bd Y$ such that $\gamma\subseteq\bd Y$.
Then $a$ must intersect $Y$ (when $S\ssm \gamma$ has two components one of the two arcs must intersect $Y$).
For any boundary component $\gamma'\ne\gamma$ which
intersects $a$ or $a'$, each intersection point lies in a segment of $\gamma'$ that passes between the ``long'' boundary edges $a\times\{0\}$ and
  $a\times\{1\}$ of $R_a$, since the other two edges are on $\gamma$. Hence such a segment
 must pass through both $a$ and $a'$. 
It follows that $\lambda\intersect Y$ and $\mu\intersect Y$ must contain arcs which are
parallel to each other, which implies $\pi_Y(\mu)$ and $\pi_Y(\lambda)$ share a
component. It follows that $d_Y(\lambda,\mu) = \diam_{\calC(Y)}(\pi_Y(\lambda)\union \pi_Y(\mu)) \le 4$.

When $Y$ is an annulus with core curve $\gamma$, denote the compactified annular cover of $S$ corresponding to $Y$ by $\widehat{Y}$.
Let $a$ be an arc of $l$ spanned by three successive intersection points with
$\gamma$. Let $\hat\gamma$ be the lift of $\gamma$ to a core curve of $\widehat{Y}$. Lift
$l$ to a geodesic $\hat l$ crossing $\hat\gamma$ and connecting the components of
$\bd\widehat Y$, and let $\hat a$ be the lift of $a$ in $\hat l$ that crosses
$\hat\gamma$. The endpoints of $\hat a$ lie in lifts $\hat\gamma_1$ and $\hat\gamma_2$ of
$\gamma$ which are lines bounding disks $H_1,H_2$ which meet the components of $\bd
\widehat Y$ in arcs. If $\lambda'$ is sufficiently close to $\lambda$ it contains a leaf
$l'$ that has a lift $\hat l'$ which passes close enough to $\hat a$ that its endpoints
lie in the disks $H_1$ and $H_2$ respectively. Then $\hat l$ and $\hat l'$ are distance at
most 2 in $\calC(Y)$, since there is a regular neighborhood of $\hat a \cup H_1 \cup H_2$
whose boundary contains an arc connecting the components of $\bd\widehat Y$ which is
disjoint from both $\hat l$ and $\hat l'$.
 \end{proof}

  \subsection{End invariants}\label{subsec:endinv}

The {\em end invariants} introduced by Brock, Masur and Minsky in \cite{bmm1} are pairs of markings or laminations, 
denoted by $(\nu^-,\nu^+)$ associated to WP geodesics.
These invariants and the associated subsurface coefficients are quite instrumental in the study 
of the global geometry and dynamics of the WP metric.

Let $r:[a,b)\to\Teich(S)$ be a complete WP geodesic ray (the domain of $r$ does not extend to the end point $b$). First, an {\em ending measure} of
$r$ is a limit (in the projective measured lamination space) of distinct Bers curves $\alpha_i$ at times $t_i\to b$. Moreover, a {\em pinching curve} along $r$ is any curve with $\lim_{t\to\infty}\ell_\alpha(r(t))=0$.
Then the union of the supporting laminations of all ending measures of $r$ and all
pinching curves along $r$ is
shown in \cite{bmm1} to be a lamination, and this is
 the {\em ending lamination}  $\nu(r)$.

Now let $g:I\to \Teich(S)$ be a WP geodesic, where $I=(a,b), [a,b)$ or $[a,b]$
($a,b\in \RR\cup\{\pm\infty\}$),
and let $c$ be a point in the interior of $I$.
When $g$ extends to $b$ (including the situation that $b\in I$) the forward end invariant of $g$ is a Bers marking at $g(b)$.
Otherwise, the forward end invariant (ending lamination) of $g$ is the ending lamination of the ray $g|_{[c,b)}$ as we defined above.
We denote the forward end invariant of $g$ by $\nu^+(g)$. The backward end invariant (ending lamination) of $g$ is defined similarly considering the ray $g(-t)|_{[-c,-a)}$ and is denoted by $\nu^-(g)$. The pair $(\nu^+(g),\nu^-(g))$ is the end invariant of $g$.
 We usually suppress the reference to the geodesic $g$ and denote the end invariant by
 $(\nu^+,\nu^-)$.

\subsection{Partial \pA maps}\label{subsec:pA}

A {\em partial \pA map supported on a subsurface $X\subsetneq S$} is a map $f\in \Mod(S)$
which fixes $X$, is homotopic to the identity on $S\ssm X$, and restricts to a \pA map on
$X$.

Any \pA map $f$ on $X$ has a unique geodesic axis $A_f$ in $\Teich(X)$, by
Daskalapoulos-Wentworth \cite[Theorem 1.1]{dwwp}. For a partial \pA map supported on $X$ we
obtain a family of axes in $\calS(\bd X)$ which can be written as $A_f \times \Teich(S\ssm X)$
in the natural product structure. If $X$ is large this is again a single axis which we
continue to denote $A_f$. 

We have the following 
lemma about subsurface coefficients of points along an axis of a partial \pA map:

\begin{lem}\label{lem:dWP-dX}
 Let $g$ be an axis of a \pA map or a partial \pA map $f$ supported on a subsurface $X$. There exists
 $D\ge 1$  so that
\[ d_Y(x,y) \leq D\]
for all $x,y\in g$ and all $Y\subseteq S$ which are not $X$ itself or annuli with cores in
$\bd X$.

Moreover, for $K\geq 1,C\geq 0$ depending only on $f$ we have
\[\dw(x,y) \asymp_{K,C}  d_{X}(x,y). \]
\end{lem}
Here $U \asymp_{K,C} V$ means, as usual, that $U \le KV + C$ and $V\le K U + C$. 

\begin{proof}
  The first statement is a corollary of
  \cite[Lemma 7.4]{wpbehavior} (see also \cite[pages 120-122]{mklcc}\cite[Theorem 3.9]{convexccgp}), which states that for
  any curve $\alpha\in \calC(S)$ there is a bound $D_\alpha$ such that
\begin{equation}\label{eq:wpbehavior bound}
  d_Y(\alpha,f^n(\alpha)) \le D_\alpha
\end{equation}
  for all $n\in\ZZ$ and $Y$  not equal to $X$ or an annulus with core in $\bd X$, provided
  $\alpha$ and $f^n(\alpha)$ intersect $Y$. 

Let $\Gamma$ be the union of all curves in Bers markings $\mu(x)$ for $x\in g$. Since
$g$ is invariant under $f$ with compact quotient, we know that there is a finite subset
$\hat\Gamma\subset \Gamma$ such that $\Gamma = \union_{n\in \ZZ} f^n(\hat\Gamma)$.

Applying (\ref{eq:wpbehavior bound}) to the curves of $\hat\Gamma$ we obtain the
first inequality of the lemma.

The second inequality follows from the first one and the  distance formula (\ref{eq:distWP}), after setting
the threshold larger than $D$. 
\end{proof}

\subsection{Length-function control}\label{subsec : lf-tw control}

In this section we assemble some of the length-function controls which we will use to extract information about behavior of 
WP geodesics from the combinatorial information.

The first result is an improved version of Wolpert's Geodesic Limit Theorem \cite[Proposition 32]{wols} which 
is Theorem 4.5 of \cite{wpbehavior}. This theorem gives us a limiting picture for
sequences of bounded length WP geodesic segments, where the overall idea is that the only
obstruction to such a sequence converging to a geodesic segment in $\Teich(S)$ or in a stratum is the
appearance of high twists along short curves. 

Given a multicurve $\sigma$ denote by $\tw(\sigma)$ the subgroup of $\Mod(S)$ generated by Dehn twists about the curves in $\sigma$. 

\begin{thm} \textnormal{(Geodesic Limits)} \label{thm : geodlimit} 
Given $T>0$, let $\zeta_{n}:[0,T]\to \overline{\Teich(S)}$ be a sequence of WP
geodesic segments parametrized by arclength.  
After possibly passing to a subsequence there is a partition
$0=t_{0}<\ldots<t_{k+1}=T$ of $[0,T]$, multicurves
$\sigma_{0},\ldots,\sigma_{k+1}$ and $\hat\tau$ where $\sigma_0$ and $\sigma_{k+1}$ may be empty such that
\[
 \sigma_{i}\cap \sigma_{i+1} \equiv  \hat\tau
\]
for all $i=0,\ldots,k$,  and a piecewise
geodesic  segment
 \[\hat{\zeta}:[0,T]\to \overline{\Teich(S)},\]
 with the following properties:
\begin{enumerate}[(GLT1)]
\item\label{gl : tau} $\hat{\zeta}((t_{i},t_{i+1}))\subseteq \mathcal{S}(\hat{\tau})$ for $i=0,\ldots,k$, 
\item \label{gl : sigma}$\hat{\zeta}(t_{i})\in \mathcal{S}(\sigma_{i})$ for $i=0,\ldots,k+1$,
\item\label{gl : converge}
  There are elements $\psi_{n}\in \Mod(S)$
and $\mathcal{T}_{i,n}\in \tw(\sigma_{i}\ssm\hat{\tau})$ for $ i=1,\ldots,k$ and
$n\in\mathbb{N}$ so that, writing
\begin{equation}\varphi_{i,n}=\mathcal{T}_{i,n}\circ \ldots\circ \mathcal{T}_{1,n}\circ
  \psi_{n}\end{equation}
for $i=1,\ldots,k$, and $\varphi_{0,n} := \psi_n$, we have
 \[\lim_{n\to\infty}\varphi_{i,n}(\zeta_{n}(t))=\hat{\zeta}(t)\]
  for any $t\in[t_{i},t_{i+1}]$, where $i=0,\ldots,k$.
 \end{enumerate}
\end{thm}

Let us also recall the Non-refraction theorem of Daskalopoulos and Wentworth \cite[Theorem
  3.6]{dwwp} (see also \cite[Theorem 13]{wols}) which specifies the stratum of the
interior of the geodesic segment connecting two points in $\oT{S}$ depending on the location of the end points.

\begin{thm}\textnormal{(Non-refraction)}\label{thm:nonref}
Let $\sigma_1$ and $\sigma_2$ be two multicurves, and $\seg{x_1x_2}$ be a WP geodesic segment with $x_1\in \calS(\sigma_1)$
and $x_2\in \calS(\sigma_2)$.  Then the interior of $\seg{x_1x_2}$ is inside $\calS(\sigma_1\cap\sigma_2)$.
\end{thm}

Formally speaking one can derive the non-refraction theorem from Theorem \ref{thm :
  geodlimit}. We will actually need the following quantified variation on non-refraction
which is also a corollary of Theorem \ref{thm :  geodlimit}.

\begin{lem}\label{lem:q-nonref} 
 For any $a>0$ and $\ep_1>0$ there exists $\ep_2>0$ such that, if
  $\zeta:[-a,a]\to\Teich(S)$ is a WP geodesic segment and $\gamma$ a curve in $S$ such that
  $$
  \max_{[0,a]}\ell_\gamma\circ\zeta \ge \ep_1
  $$
  then
  $$
  \max_{[-a,0]}\ell_\gamma\circ\zeta \ge \ep_2.
  $$
\end{lem}

\begin{proof}
  Supposing the lemma fails, there is a sequence of geodesics $\zeta_n:[-a,a]\to\Teich(S)$ and curves $\gamma_n$ such that
  $\ell_{\gamma_n}(\zeta_n(t_n)) \ge \ep_1$ for some $t_n\in [0,a]$ while
  $  \max_{[-a,0]}\ell_{\gamma_n}\circ\zeta_n \to 0$ as $n\to\infty$.

Note, by convexity of length-functions we may assume that $t_n = a$. 
  
Use Theorem \ref{thm : geodlimit} (GLT)  to obtain (passing to a subsequence if necessary)  a
partition $t_0,\ldots,t_{k+1}$ of $[-a,a]$, multicurves $\sigma_0,\ldots,\sigma_{k+1}$ and
$\hat \tau$, mapping classes $\psi_n, \mathcal{T}_{i,n}$ and a piecewise geodesic $\hat\zeta$
satisfying the conclusions of the theorem. In particular, by GLT3, $\psi_n(\zeta_n) \to \hat\zeta$ on
the interval $(a,t_1)$ as $n\to\infty$, and this limit by GLT1 lies in $\calS(\hat\tau)$. So we conclude
$\psi_n(\gamma_n)$ is eventually a component $\gamma$ of $\hat\tau$,
and hence $\ell_\gamma\circ \hat{\zeta}\equiv 0$ on $[-a,a]$.

Now since $\hat\tau\subseteq \sigma_i$ for each $i$, each $\mathcal{T}_{i,n}$, which
is in the twist group of $\sigma_i$, must fix $\gamma$. This means that $\varphi_{i,n}(\gamma_n)=\gamma$
(with $\varphi_{i,n}$ defined as in (GLT3)), and since by GLT3, $\varphi_{i,n}(\zeta_n) \to \hat\zeta$ on
$[t_i,t_{i+1}]$, we conclude that $\ell_{\gamma_n}\circ\zeta_n\to
\ell_\gamma\circ\hat\zeta$ on
$[-a,a]$. But $\ell_\gamma\circ \hat{\zeta}\equiv 0$ so we have a contradiction to the lower bound for $\ell_{\gamma_n}(\zeta_n(t_n))$. 
\end{proof}

The following two results from \cite[\S 4]{wpbehavior}, obtained there as consequences of
Theorem \ref{thm : geodlimit},  provide us with 
control of the length of a curve along WP geodesics in terms of the associated annular
coefficient of the curve. Roughly speaking, along a WP geodesic segment of bounded length
with suitable assumptions at the endpoints, the length of a curve $\gamma$ gets very short
somewhere in the middle if and only if the twisting of the endpoints around $\gamma$ grows
very large. 

\begin{thm}\label{thm : twsh}\tn{(\cite[Corollary 4.10]{wpbehavior})}
Given $T, \epsilon_{0}$ and $\epsilon<\epsilon_{0}$ positive, there is an $N\in\NN$ with
the following property. Let  $\zeta: [0, T'] \to \Teich(S)$ be a WP geodesic segment
parameterized by arclength with
$T'\leq T$ such that for a curve $\gamma$ 
\[\max_{t \in [0,T']}\ell_{\gamma}(\zeta(t))> \epsilon_{0}\] 
and
\[d_{\gamma}(\zeta(0), \zeta(T')) \geq N.\]
Then we have
  \[\min_{t \in [0,T']} \ell_{\gamma}(\zeta(t))< \epsilon.\]
\end{thm}

\begin{thm}\label{thm : shtw}\tn{(\cite[Corollary 4.11]{wpbehavior})}
Given $\epsilon_{0},T,s$ positive with $T>2s$ and $N\in\NN$, there is an $\epsilon<\epsilon_{0}$ with the following property. Let $\zeta: [0, T'] \to \Teich(S)$ be a WP geodesic segment parametrized by arclength with $T'\in [2s,T]$. Let $J \subseteq [s,T'-s]$ be a subinterval, and suppose that for some $\gamma\in\mathcal{C}_{0}(S)$ we have
\[\max_{t \in [0,T']} \ell_{\gamma}(\zeta(t))> \epsilon_{0}\]
and
\[\min_{t \in J} \ell_{\gamma}(\zeta(t)) < \epsilon.\]
Then we have
\[d_{\gamma}(\zeta(0),\zeta(T'))\geq N.\]
\end{thm}

We will need the following variant on Theorem \ref{thm : shtw} as well:

\begin{thm}\label{thm : shtw2}
Given $\epsilon_{0},T>0$  and $N\in\NN$, there is an $\epsilon<\epsilon_{0}$ with the
following property. Let $\zeta: [0, T'] \to \Teich(S)$ be a WP geodesic segment
parametrized by arclength with $T'\le T$. Suppose that for some $\gamma\in\mathcal{C}_{0}(S)$ we have 
\[\ell_\gamma(\zeta(0))> \ep_0, \qquad  \ell_\gamma(\zeta(T'))> \ep_0\]
and
\[\min_{t \in [0,T']} \ell_{\gamma}(\zeta(t)) < \epsilon.\]
Then we have
\[d_{\gamma}(\zeta(0),\zeta(T'))\geq N.\]
\end{thm}

\begin{proof}
First we quote the following direct consequence of \cite[Corollary 3.5]{wpbehavior}:
\begin{lem}\label{lem:distlen}
For any $l>0$ there is an $s>0$ so that if for a curve $\beta\in \calC(S)$, $\ell_\beta(x)\geq l$,
then for all $x'\in \Teich(S)$ with $d(x,x')\leq s$ we have that $\ell_\beta(x')\geq l/2$.
\end{lem}

Now let $s>0$ be the constant given by this lemma for $l=\ep_0$, and let $\ep$ be the
constant from Theorem \ref{thm : shtw} given $\ep_0,T$ and $s$. Now if $\ep' <
\min\{\ep,\ep_0/2\}$ we find that, if $\ell_\gamma(\zeta(t)) < \ep'$ then $T'>2s$ and $t\in
J=[s,T'-s]$. Therefore Theorem \ref{thm : shtw} applies to give us the desired
conclusion. 
\end{proof}

The following theorem which relies on convexity of length-functions
provides us with conditions for approach to strata or
having short curves along WP geodesics
  (see also \cite[Lemma 6.9]{wpbehavior}).

\begin{thm}\label{thm:lenapproach0}
 Let $c_1,c_2>0$, and let $\sigma$ be a co-large multicurve. 
Let $\zeta_n: I_n\to \Teich(S)$ be a sequence of WP geodesic segments, where $I_1\subset
I_2\subset\cdots$ and $\cup_n I_n = \RR$. 
Suppose that $\ell_\alpha\circ \zeta_n(t)<c_1$ for all $\alpha\in\sigma$ and $t\in I_n$.
Let $J$ be a compact interval for which
$\ell_\beta\circ \zeta_n(t)>c_2$ for all $\beta$ disjoint from $\sigma$ and $t\in J\cap I_n$. Then,
after possibly passing to a subsequence,  for all $\alpha\in\sigma$ we have 
\[\ell_\alpha\circ \zeta_n\to 0\]
uniformly on $J$ as $n\to\infty$
\end{thm}

 \begin{proof}
First we record the following elementary fact. In the following for an interval $I =[a,b]$ 
we denote by $\half I$  the interval with the same center and half the diameter.
\begin{lem}\label{lem:convexity}
Let $f:I \to \RR$ be a $C^2$ function that satisfies $0\le f \le c_1$ and $\ddot f > 0$. Then 
\begin{equation}\label{eq : f dot bound}
  | \dot f(t)| \le 4c_1/|I|
\end{equation}
  for all $t\in\half I$.
  \end{lem}
  \begin{proof}
This is an exercise in calculus. For $t\in \half I$, if $\dot f(t) \ge 0$ we have
  \[
  c_1 \ge f(b)-f(t) = \int_t^b \dot{f}(s)ds \ge (b-t)\dot f(t),
  \]
  because $\dot f$ is increasing.
  Then since $b-t \ge |I|/4$ we have $\dot f(t) \le 4 c_1/|I|$. 
  When $\dot f (t) \le 0$ the argument is similar (integrating on $[a,t]$).
\end{proof}  
  
Now let $\alpha\in \sigma$. Note for large enough $n$ that $J\subset \half I_n$. Since 
$\ell_\alpha\circ \zeta_n$ is bounded by $c_1$ on $I_n$ and
$\ddot{\ell}_\alpha\circ\zeta_n(t)>0$ for all $t\in I_n$ (\cite[Corollary 4.7]{wolpert:nielsen}), Lemma
\ref{lem:convexity} applies and  gives us
\[|\dot{\ell}_\alpha\circ \zeta_n| = O(1/|I_n|)\]
 on $J$ for all $n$ large enough.
Since $|I_n|\to \infty$, we have $\dot{\ell}_\alpha\circ\zeta_n \to 0$ uniformly on $J$,
and thus, possibly passing to a subsequence, we have a constant $c_\alpha\ge 0$ for each
$\alpha\in \sigma$ such that $\ell_\alpha\circ\zeta_n \to c_\alpha$ uniformly on
$J$.

Partition $\sigma$ into $\tau\sqcup \kappa$, where $c_\alpha = 0$
for $\alpha\in \kappa$ and $c_\alpha > 0$ for $\alpha\in\tau$. 
Now, note that by the assumption of the theorem the length of any curve disjoint from $\sigma$
is bounded below by $c_2$, and by the Collar Lemma \cite[\S 4.1]{buser} the length of every curve that overlaps $\sigma$ is at least 
the size of the standard collar \nbhd of a curve with length at most $c_1$.
Thus, the only curves whose lengths go to $0$ on $J$ are the ones in $\kappa$.
In the following we show that $\tau$ is empty, which means $\kappa=\sigma$,
 and hence the lengths of all curves in $\sigma$
converge to $0$ uniformly on $J$ as is desired.

Seeking a contradiction suppose $\tau\neq\emptyset$. By compactness of the completed
moduli space we may assume, up to composing $\zeta_n$ by mapping classes in $Stab(\sigma)$
and passing to a subsequence, that $\zeta_n|_J$ converge pointwise to a geodesic $\zeta$.
Since $\ell_\alpha\circ\zeta_n|_J\to 0$ for all $\alpha\in\kappa$,  
and the length of every curve which is not in $\kappa$ is bounded below together with 
 the fact that length-functions extend continuously to the WP completion of \T space
  imply that
$\zeta$ lies in the stratum $\calS(\kappa)$. 
Hence for an $\alpha\in \tau$, the function $\ell_\alpha\circ \zeta_n$ converges to
$\ell_\alpha\circ\zeta$ pointwise, which implies $\ell_\alpha\circ\zeta \equiv
c_\alpha>0$ on $J$. 
But $\kappa$ is a co-large multicurve (see Remark \ref{remark:sub co-large}) and hence
$\calS(\kappa)\cong \Teich(S\ssm \kappa)$, 
so again by \cite[Corollary 4.7]{wolpert:nielsen} the function $\ell_\alpha\circ \zeta$ is strictly convex.
This contradiction shows that $\tau$ is empty, and
completes the proof of the theorem.
\end{proof}

We also will use the following result which was proved in the setting of hierarchy paths
in \cite[Lemma 6.4]{wpbehavior}:

\begin{thm}\label{thm:anncoeffcomp}
For any $k,K\geq1,c, C\geq 0$ and $D\geq 0$, there exist constants $w, B>0$ so that the following hold: 
Let $\rho:I\to\calP(S)$ be a $(k,c)$--quasi-geodesic with the property that for a non-annular subsurface 
$X\subseteq S$, and any $i,j\in I$ we have
\begin{equation}\label{eq:advinX}
d_{\calP}(\rho(i),\rho(j))\asymp_{K,C} d_X(\rho(i)),\rho(j)).
\end{equation} 
Moreover, let $\gamma$ be a curve with $\gamma\pitchfork X$ which is in the pants decomposition $Q$ where $d(Q,\rho(i))\leq D$, and
 let $P\in \calP(S)$ be so that $d(P,\rho(j))\leq D$ for a $j\in I$ where $|j-i|\geq w$.
Then we have that
  \begin{equation}\label{d gamma bounded}
 d_\gamma(P,\rho(j))\leq B
  \end{equation}
 Moreover, if $P$ is a Bers pants decomposition at a point $x\in \Teich(S)$, then we have
  \begin{equation}\label{eq:lb for gamma}
 \ell_\gamma(x)\geq \omega 
   \end{equation}
   where $\omega>0$ is the width of the standard collar \nbhd of a Bers curve on $x$.
\end{thm}
\begin{proof}
Here we just sketch the proof and refer the reader to \cite[Lemma 6.4]{wpbehavior}. The assumption (\ref{eq:advinX}) 
and the fact that $\rho$ is a quasi-geodesic imply that $\rho$
 advances at a definite rate in the curve complex $\calC(X)$. This together with
 the Lipschitz property of $\pi_\gamma: \calC(X)\to \calC(\gamma)$ outside a bounded
 neighborhood of $\gamma$ in $\calC(X)$ show that:
choosing $w$ large enough for $i,j\in I$ and a $P\in \calP(S)$ satisfying assumptions of the lemma 
 a shortest path of length $D$ that connects $P$ and $\rho(j)$ in $\calP(S)$ passes only through pants decompositions that 
intersect $\gamma$. Thus the inequality (\ref{d gamma bounded}) holds for a $B>0$ depending only on $D$.

 Moreover, the inequality (\ref{eq:lb for gamma}) follows from the fact that $\gamma$ intersects a Bers curve at $x$, 
 and hence by the Collar Lemma (\cite[\S 4.1]{buser})
 the length of $\gamma$ is at least $\omega>0$ the width of the standard collar \nbhd of a Bers curve.
\end{proof}

\subsection{Ruled surfaces in \W metric}
\label{subsec:rulings}

In this subsection we assemble some facts and results about ruled
surfaces in the \W metric, mainly drawn from $\S 4,6$ of \cite{asympdiv}.

Let $\zeta$ be a WP geodesic in $\Teich(S)$ and
$\pi_\zeta:\Teich(S)\to \zeta$ the nearest-point projection. This map
is smooth at points $x$ with $\pi_\zeta(x)$ in the interior of
$\zeta$ \cite[Proposition 4.1]{asympdiv}. 
We will consider ruled surfaces over $\zeta$ as follows: 

Let $\eta$ be a path in $\Teich(S)$. 
Then the geodesic segments
$\seg{x\pi_\zeta(x)}$ for $x\in\eta$ form a ruled surface which we denote
$Q[\eta;\zeta]$.  Given a parameterization of $\eta$ by arclength,
written $\eta:[0,T] \to \Teich(S)$, we can parameterize $Q[\eta;\zeta]$ as
$$
Q : \Delta_Q \to \Teich(S)
$$
where $\Delta_Q$ is the planar region 
\[
\Big\{ (t,s): t\in [0,T], s\in [0,\lambda(t)] \Big\}
 \]
 and $\lambda(t)$ is $d(\eta(t),\zeta)$,
 which is the length of $\overline{\eta(t)\pi_\zeta(\eta(t))}$. 
Thus $t\mapsto Q(t,\lambda(t))$ parametrizes $\eta$ and
$t\mapsto Q(t,0)$ parameterizes $\pi_\zeta\circ\eta$. For each $t$,
$s\mapsto Q(t,s)$ parametrizes
$\overline{\eta(t)\pi_\zeta(\eta(t))}$ by arclength. See Figure \ref{Q-param}.

\realfig{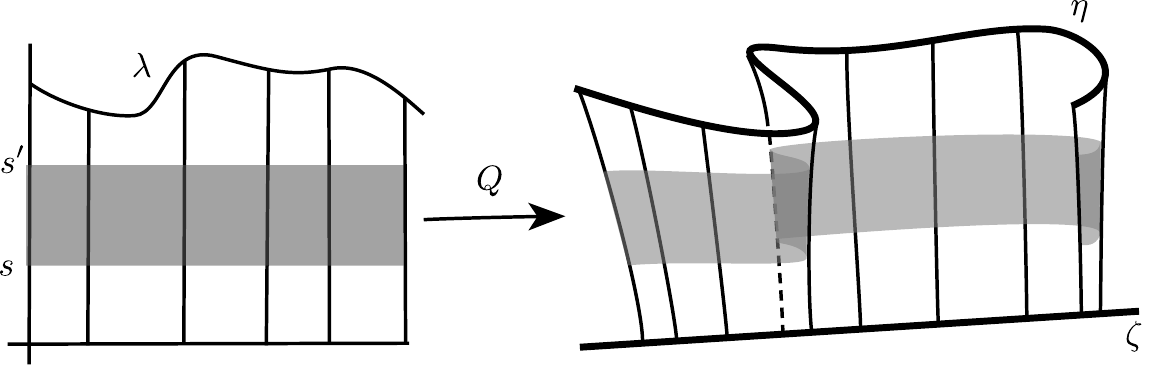}{A ruled surface $Q[\eta;\zeta]$ over a geodesic $\zeta$. The shaded
  region corresponds to $Q_s^{s'}[\eta;\zeta]$.}

Regions inside $Q$: Note that $Q(t,s)$ is at distance $s$ from
$\zeta$. Thus, for $0\le s < s' < \dist(\eta,\zeta)$ we can restrict
$Q$ to $[0,T]\times[s,s']$ and we denote this by
$Q_s^{s'}[\eta;\zeta]$ or just $Q_s^{s'}$. Similarly we denote by
$Q^s$ the level curve $Q$ restricted to $[0,T]\times s$. Note that
$Q^0$ is a (not necessarily injective) parametrization of $\zeta$.

If $\pi_\zeta(\eta)$ is contained in the interior of $\zeta$, and if
$\eta$ is smooth, then $Q[\eta;\zeta]$ is smooth and 
for each $x\in \eta$ the geodesic segment $\overline{x\pi_\zeta(x)}$
is orthogonal to $\zeta$. In fact the level curves $Q^s$ are
orthogonal to the rulings $\seg{x\pi_\zeta(x)}$ at all intersection points.

We can define the pullback metric on $Q[\eta;\zeta]$ (or on the
parametrizing domain) and denote its Gaussian curvature $\kappa$.
We also define the intrinsic geodesic curvature $k_g$
along horizontal curves $Q^s$ for $s>0$, oriented so that is positive if $Q^s$
is curved away from the bottom curve $Q^0$. With this convention
$-k_g$ is non-negative \cite[Theorem 4.2]{asympdiv}. We let $dm_s$ be the
measure on the level curve $Q^s$ induced by integrating $-k_g$.
Negative curvature implies that this family of measures is monotonic
and  weak-$*$ converges to a measure $m$ on $Q^0$ 
\cite[Claim 4.3]{asympdiv},
in particular
$$
\int_{Q^0} dm = \lim_{s\to 0} \int_{Q^s} -k_g. 
$$
This provides the curvature term for the
bottom edge of $Q$ which allows us to write a version of the
Gauss-Bonnet theorem (this is formula (4.5) in \cite[\S
  4.1]{asympdiv}): 
\begin{thm}\label{thm : GB2}\tn{(Gauss-Bonnet)} 
Let $Q=Q[\eta;\zeta]$ where $\eta$ is  also a geodesic, and
let $\theta_i,\; i=1,2,3, 4$ be the exterior angles at the four corners of $Q$. Then
 \[\iint_{Q} \kappa dA -\int_{Q^0} dm = 2\pi-\sum_{i=1}^4\theta_i.\]
\end{thm}
\begin{remark}
Note that the exterior angles at the bottom corners are $\frac{\pi}{2}$,
since we are in the case where the rulings of $Q$ are orthogonal to
$\zeta$. However, the internal angles at those corners might be larger
than $\frac{\pi}{2}$, which would be accounted for by atomic components of
the measure $dm$.
\end{remark}

It is helpful to define, for any $Q$, 
\begin{equation}\label{def: I}
\calI(Q) = \iint_{Q}-\kappa dA + \int_{Q^0}
dm
\end{equation}
Note that $\calI(Q) \ge 0$ and is monotonic, so that for example if $s
\le dist (\eta,\zeta)$ and $\eta'$ is a sub-path of $\eta$, we have
\begin{equation}\label{I monotonic}
\calI(Q_0^s[\eta';\zeta]) \le \calI(Q[\eta;\zeta]).
\end{equation}

\subsubsection*{Lower bounds on $\calI$}
When $\zeta$ is close to the thick part of the stratum of a co-large multicurve, we
obtain lower bounds on $\calI(Q)$ for certain ruled surfaces $Q$.
Recall large subsurfaces from Definition \ref{def:large}.
Then, Lemma 6.3 in \cite{asympdiv} importing some of the information from the statement of
Theorem 5.14 of the paper can be rephrased as follows:

\begin{lem}\label{lem:I lower bound}

 Let $\bar\ep>0$ and let $b>0$ be the corresponding constant from Lemma \ref{lem : b-nbhd}.
Then, for any $d\in (0,b)$ and $e> 0$
 there exists a $K_0>0$ such that the
 following holds:  Let $\sigma$ be a co-large multicurve, 
and let $\zeta$ be a geodesic segment in
  $U_{d,\bar\ep}(\sigma)$ with length at least $1$. Let $Q[\eta;\zeta]$ be a
  ruled surface with $\dist(\eta,\zeta)>e$. Then
  \[
  \calI(Q_0^e) \ge K_0.
\]
\end{lem}

\begin{proof}
Here we only sketch the proof of the lemma. The detailed analysis, using suitable frame fields introduced by Wolpert \cite[\S 4]{wolb}, and standard
properties of Jacobi fields is carried out in $\S 5$ and Lemma 6.3 of \cite{asympdiv}.

Negative curvature implies that the level sets $Q^t$ are expanding with $t$, so that the
area of $Q_0^e$ is bounded below by $e$. Hence the first term in the definition of
$\calI(Q_0^e)$ would give the desired lower bound provided that the sectional curvatures
(in the planes tangent to $Q$) are bounded away from 0. These sectional curvatures are
indeed strictly negative in the thick part of Teichm\"uller space, as well as, near the
stratum of $\sigma$, in the directions nearly tangent to the stratum (this last fact
follows from the assumption that $S\ssm \sigma$ is large, hence the stratum has no
nontrivial product structure). Thus one may consider, pointwise on $Q$, two cases:
if the ruling geodesic direction of $Q$ is nearly tangent to the stratum direction, one
obtains a strictly negative curvature bound. If the ruling is transverse to
the stratum direction, then in one direction or the other the ruling geodesic exits a
neighborhood of the stratum, and enters the regime of strictly negative curvature in all
directions. This again gives a definite contribution to the integral. 
\end{proof}

\subsection{Asymptotic rays}
\label{subsec:asymptotic}

In this subsection we prove a result on asymptotic and strongly asymptotic rays that will
be useful in Section \ref{bottlenecks}. 
The first statement of the proposition is a variation on Theorem 6.2 of \cite{asympdiv},
giving a criterion for promoting asymptoticity to strong asymptoticity in our setting.  

(Recall that a ray $r:[0,\infty)\to \calX$ in a metric space is {\em asymptotic} to a subset $\calY\subseteq \calX$
 if $\dist(r(t),\calY)$ is bounded above for all $t\in [0,\infty)$.
The ray is {\em strongly asymptotic} to the subset if
$\lim_{t\to\infty}\dist(r(t),\calY)=0$.)

\begin{prop}\label{prop : asymp-strongasymp}
  Let $r$ and $r'$ be two asymptotic geodesic rays in $\oT{S}$
  such that
  $r([T,\infty])\subset U_{d,\bar\ep}(\omega)$ 
  for $d<b$,  $T\ge 0$ and a co-large multicurve $\omega$. Then $r$ and $r'$ are in fact
  strongly asymptotic.

  For any ray $r$ contained in $\calS_{\bar\ep}(\omega)$ there is a ray $r_1$ in $\Teich(S)$ which
  is strongly asymptotic to $r$. 
\end{prop}

\begin{proof}
We begin by proving the first statement in the case where $r$ and $r'$ are in the
interior, $\Teich(S)$. 
  
  Following the notation of $\S$\ref{subsec:rulings} let $Q$ be the ruled surface $Q[r';r]$.
  Let $\pi_r$ denote the nearest point projection to $r$. Since it is continuous, for each
  $n$ there is an interval $J_n$ such that $\pi_r(r'(J_n))= r([0,n])$. Similarly for each
  $i$ there is an interval $I_i$ such that $\pi_r(r'(I_i))= r([i,i+1])$. 

  Suppose by way of
  contradiction that the distance from $r'(t)$ to $r$ remains bounded below by $e>0$ for all
  $t$. Then, since $r([i,i+1])\subset U_{d,\bar\ep}(\omega)$, 
  by Lemma \ref{lem:I lower bound} we have
  $\calI(Q_0^e[r'|_{I_i},r|_{[i,i+1]}]) \ge K_0$ for a uniform $K_0>0$, and hence
  \[\calI(Q[r'|_{J_n},r|_{[0,n]}) \geq K_0n.\] 
    Moreover, by Theorem \ref{thm : GB2} (Gauss-Bonnet) the left-hand side of the above inequality is bounded above by $4\pi$, which implies that $n\leq \frac{4\pi}{K_0}$. 
    However $n$ could be chosen arbitrarily large which is a contradiction. Therefore, $r$ and $r'$ are strongly asymptotic.

    Now we prove the second part. Let $r$ lie in $\calS_{\bar\ep}(\omega)$, where
    $\omega$ is co-large. 
As in the proof of Theorem 1.3 of \cite{bmm1}, we fix a basepoint $x\in \Teich(S)$
within $b$ of $r(0)$, let $y_n = r(n)$ and use CAT(0) geometry of $\oT{S}$
(via Lemma 8.3 in \cite[\S II.8]{bhnpc})
to conclude that the segments $\seg{xy_n}$ converge
to an infinite ray $r_1$ in $\oT{S}$, which is asymptotic to $r$.

We claim that $r_1$ is entirely inside $\Teich(S)$. Suppose that $T>0$ is the first time
that $r_1$ intersects a completion stratum $\calS(\sigma)$. The segment $r_1([0,T+1])$
then has at least one endpoint in $\Teich(S)$ and hence by 
Theorem \ref{thm:nonref} (Nonrefraction) its interior maps to $\Teich(S)$. This
contradicts the assumption that $r_1(T)\in\calS(\sigma)$.

\medskip
To see that $r_1$ is strongly asymptotic to $r$, 
for a $\delta>0$ let $x_\delta$ 
be the point on the geodesic segment $\seg{r(0)r_1(0)}$ at distance $\delta$ from $r(0)$.
(it must be in $\Teich(S)$ by Theorem \ref{thm:nonref}).
Let $r_\delta$ be the geodesic obtained as above as the limit of $\seg{x_\delta y_n}$.
As we saw above $r_\delta$ is an infinite ray in $\Teich(S)$ which 
remains in a $\delta$-neighborhood of $r$ and in particular in $U_{\delta,\bar\ep}(\omega)$.
Since the rays $r_\delta$ and $r_1$ are asymptotic and in $\Teich(S)$, the first part of the proposition
implies (for sufficiently small $\delta$) that $r_1$ and $r_\delta$ are in fact strongly asymptotic.
Letting $\delta\to 0$, the strong asymptoticity of $r_1$ to $r$ follows.

Finally, we prove the first part of the proposition in the general case. Using the second
part, $r$ and $r'$ are strongly asymptotic to $r_1 $ and $r'_1$ respectively, which lie in the
interior. The version of first part that we already proved shows that $r_1$ and $r'_1$ are
strongly asymptotic, and this concludes the proof. 
\end{proof}

\section{Non-annular bounded combinatorics}\label{sec:nonannularbddcomb}

In this section we study the case that the end invariant of a WP geodesic satisfies the
{\em non-annular bounded combinatorics} condition:

\begin{definition}(Bounded combinatorics)
We say that a pair of markings or laminations $(\mu,\mu')$ satisfies {\em $R$--bounded combinatorics}  if
\[d_Y(\mu,\mu')\leq R\]
for all proper subsurfaces $Y\subsetneq S$. If the bound holds for non-annular proper subsurfaces, we say the  
pair satisfies {\em non-annular} $R$--bounded combinatorics.
\end{definition}

When the end invariant of a WP geodesic satisfies non-annular bounded combinatorics, the
short curves correspond exactly to the annulus with big subsurface
coefficients. More precisely:

\begin{thm}\label{thm:short curve bddcomb}
  For any $R,\ep_0>0$ there are functions $\hat N:\RR_{>0}\to\RR_{>0}$ and
  $\hat \ep:\RR_{>0}\to\RR_{>0}$ such that the following holds. 

  Suppose that $g$ is a WP geodesic with end invariant
  $(\nu^+,\nu^-)$, where $\nu^\pm$ are either laminations in $\EL(S)$ or
  points in the $\ep_0$-thick part, 
which satisfy non-annular $R$--bounded combinatorics.
Then 
\begin{enumerate}
\item for any $N\geq 1$, if $\inf_t\ell_\gamma(g(t))< \hat\ep(N)$ then $d_\gamma(\nu^-,\nu^+)\geq N$.
\item for any $\ep>0$,  if $d_\gamma(\nu^-,\nu^+)\geq \hat N(\ep)$ then
  $\inf_t\ell_\gamma(g(t))< \ep$.
\end{enumerate}
\end{thm}

\begin{proof}
Let $\rho:I\to \calP(S),\; I\subseteq \RR\cup \{\pm\infty\}$, be a hierarchy path in the
pants graph of $S$ that connects the points $\nu^-$ and $\nu^+$; see
\S \ref{subsec:coarse-model-CC}.  Let $Q:\Teich(S)\to \calP(S)$ be Brock's (see (\ref{eq:brock})).
The non-annular $R$--bounded combinatorics property implies, via
 \cite[Theorem 4.4]{bmm2} (also \cite[Theorem 5.13]{wpbehavior}), that 
$Q\circ g$  and $\rho$ are $D$--fellow-travelers in $\calP(S)$,
 where $D$ is a constant depending only on $R$. Moreover, the non-annular $R$-bounded combinatorics condition
together with the distance formula (\ref{eq:dist}) 
implies
  that condition (\ref{eq:advinX}) in Theorem \ref{thm:anncoeffcomp} holds with the
  whole surface $S$ playing the role of $X$.    That is,
$$ d_{\calP}(\rho(i),\rho(j))\asymp_{K,C} d_S(\rho(i)),\rho(j))$$ 
with constants $K,C$ depending only on $R$.

\subsubsection*{Proof of part (1):}
Given $\ep>0$ and less than the Bers constant $L_S$, suppose for some $t\in \RR$ that
$\ell_\gamma(g(t))< \ep$.
Then $\gamma$ is in a Bers pants decomposition $Q(g(t))$ at $g(t)$,
 and by the fellow traveling of $Q(g)$ and $\rho$, $Q(g(t))$ is within distance $D$ of
 $\rho(i)$ for some $i\in I$. 

Let $w,B>0$ be the constants provided by  Theorem \ref{thm:anncoeffcomp}. Then let $t_1,t_2$ be so that the pants decompositions
 $Q(g(t_1))$, $Q(g(t_2))$ are within distance $D$ of $\rho(i-w)$ and $\rho(i+w)$,
respectively, then the inequality (\ref{d gamma bounded}) from the theorem gives us

\begin{equation}\label{eq:anncoeff}
\Big|d_\gamma\left( Q(g(t_1)),Q(g(t_2)) \right) - d_\gamma\left(\rho(i-w),\rho(i+w)\right)\Big| \le 2B.
\end {equation}
Moreover, by (\ref{eq:lb for gamma}) from the theorem we have that:
\begin{equation}\label{eq:lb}
\min\{\ell_\gamma(g(t_1)),\ell_\gamma(g(t_2))\}> \omega.
\end{equation}
 
 Now, note that the length of $[t_1,t_2]$ is bounded independently of $g$ and $t$ with a constant that depends only on $R$ and $D$.
Thus we can apply
Theorem \ref{thm : shtw2} to the geodesic segment $g|_{[t_1,t_2]}$ to
conclude that, for any $N\in \NN$, there is an $\ep<\min\{\omega, L_S\}$, so that if $\inf_{t\in [t_1,t_2]}\ell_\gamma(g(t))<\ep$, then
\[d_\gamma(Q(g(t_1)),Q(g(t_2)))\geq N+2B+M,\]
where the constant $M$ is from Proposition \ref{prop:nobacktr} (no-backtracking).
By the inequality (\ref{eq:anncoeff}) this implies that
  \[d_\gamma(\rho(i-w),\rho(i+w))\geq N+M.\] 
Then by Proposition \ref{prop:nobacktr} we have
  \[d_\gamma(\nu^-,\nu^+)\geq N\]
which concludes the proof of part (1).

\subsubsection*{Proof of part (2):}
By \cite[Lemma 6.2]{mm2} if $d_\gamma(\nu^-,\nu^+)\geq N$ where $N\in \NN$ is larger than a threshold,
the curve $\gamma$ appears as a
curve in a pants decomposition $\rho(i)$ where $i\in I$. 
Then similar to part (1) by Theorem \ref{thm:anncoeffcomp}
 there are constants $w,B>0$ and $\omega>0$, so that the inequalities
(\ref{eq:anncoeff}) and (\ref{eq:lb}) hold.

Moreover, appealing again to 
the no-backtracking property of hierarchy paths, we have 
$$d_\gamma(\rho(i-w),\rho(i+w))\geq N-M.$$
The fellow-traveling property of $\rho$ and $Q(g)$ guarantees that there are $t_1,t_2$ so that $Q(g(t_1))$ and
$Q(g(t_2))$ are within distance $D$ of $\rho(i-w)$ and $\rho(i+w)$
respectively. So by (\ref{eq:anncoeff}) we have
\[d_\gamma(Q(g(t_1)),Q(g(t_2)))\geq N-M-2B.\]
Also, the length of the interval $[t_1,t_2]$ is bounded independently of $g$, 
thus appealing to Theorem \ref{thm : twsh}, for any $\ep>0$, there is an $N\in \NN$ so that
  \[\inf_{t\in [t_1,t_2]}\ell_\gamma(g(t))< \ep,\]
which gives us part (2) of the theorem.
\end{proof}

\section{Bottlenecks and visibility}
\label{bottlenecks}

The main result of this section is Theorem \ref{thick bottleneck} on existence of
bottlenecks, which we restate here:

\restate{Theorem}{thick bottleneck}{
Let $\omega,\omega'$ be two co-large multicurves that fill $S$. 
Let $\bar\ep>0$ and let $r$ and $r'$ be infinite length WP geodesic rays in $\oT{S}$
that are strongly asymptotic to (or contained in) the $\bar\ep$--thick parts of the
strata $\calS(\omega)$
and $\calS(\omega')$, respectively. 
Then $r$ and $r'$ have a bottleneck.
}

As a consequence of the above theorem we will also obtain a visibility theorem, which is
Theorem \ref{thm:asymplargevis} stated in Subsection \ref{subsec:visibility}. 

We start with the following observations about the rays $r$ and $r'$:
\begin{lem}\label{rays diverge}
 The rays $r$ and $r'$ diverge; that is
\[
\lim_{t\to\infty}\dist(r'(t),r) = \infty
\]
and the corresponding statement holds when interchanging $r$ and $r'$.
\end{lem}

\begin{proof}
The lemma would follow once we show that, for any $R>0$,  the intersection of $R$--neighborhoods
$\calN_R(\calS(\omega))\cap \calN_R(\calS(\omega'))$ has finite
diameter.

To see this fact, we start with a consequence of the distance formula (\ref{eq:distWP}):
\begin{equation}\label{stratum distance formula}
\dist(x,\calS(\omega)) \asymp \sum_{\substack{Y\subseteq S: \;\na ,\; \omega\pitchfork Y}} \{d_Y(x,\omega)\}_A
\end{equation}
with constants that depend only on $A$. This follows by checking that, for any $y\in \calS(\omega)$
minimizing pants distance to $x$, the projections to subsurfaces in the complement of
$\omega$ do not contribute to the sum. The argument appears, in a slightly different
context, in Proposition 3.1 of \cite{bkmm}.

Now if $x\in \calN_R(\calS(\omega))\cap \calN_R(\calS(\omega'))$ we apply (\ref{stratum distance formula}) to both $\omega$ and $\omega'$. 
Since $\omega\union\omega'$ fills $S$, every $Y\subseteq S$ intersects $\omega$ or $\omega'$, so we obtain an upper bound independent of $x$ for
\[
\sum_{\substack{Y\subseteq S:\;\na }} \{d_Y(x,\omega\union\omega')\}_A.
\]
But since $\omega\union\omega'$ is a fixed collection of curves this gives a uniform upper bound on $d(x,x_0)$ for a fixed basepoint $x_0$ and all
$x\in \calN_R(\calS(\omega))\cap \calN_R(\calS(\omega'))$. This gives the desired diameter bound. 
\end{proof}

For the rest of the proof let us assume that $r $ and $r'$ are in the interior
$\Teich(S)$. At the end we will derive the full statement.

\subsection{The ruled triangle argument}

 Let $u = r(0)$ and $u' =
r'(0)$. We will show that
for any two points $p$ on $r$ and $q$ on $r'$ the geodesic segment
$\seg{pq}$ meets a compact subset of \T space. The first ingredient of the proof is the following:

\begin{lem}\label{lem:distance bound}
 The distance $\dist(u,\seg{pq})$ is bounded independently of $p$ and $q$.
\end{lem}

\begin{proof}
  Let $\zeta$ denote the composed path $\oseg{u q} * \oseg{qp}$.
  Then let $Q = Q[\zeta;r]$ be the ruled surface over $r$
  parametrized by $\zeta$ (as defined in \S\ref{subsec:rulings}).

  Let $b>0$ be the constant from Lemma \ref{lem : b-nbhd} corresponding to $\bar\ep>0$,
and let $J\subset \zeta$ be the subset of $\zeta$ at distance greater
  than $b/2$ from $r$ i.e.
  \begin{equation}\label{eq:J}
  J := \{v\in\zeta: d(v,r) \ge b/2\}.
  \end{equation}
By convexity of the distance function in a CAT(0) metric the interval $J$, if nonempty, 
is an interval containing the apex $q$ of the geodesic triangle $\triangle u
  q p$. 

  \realfig{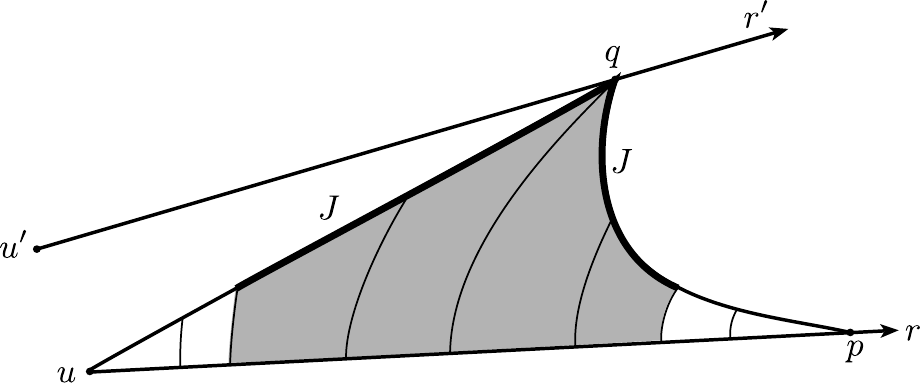}{The ruled triangle $Q[\zeta;r]$ determined by
    $u,p,q$. The ruled subsurface $Q[J;r]$ is shaded.}

  \begin{claim}\label{claim:compact J}
    There is a bounded interval $I$ of $r$ containing 
   $\pi_r(J)$   for any $p,q$. 
  \end{claim} 
  
  To see this, we first show that the left endpoint of $\pi_r(J)$ is a
  bounded distance from $u$.

  We know from Lemma \ref{rays diverge} that $\dist(r'(t),r)\to
  \infty$. Thus let $t_0$ be such that $t>t_0$ implies $\dist(r'(t),r) > b/2+D$, where  $D = d(u,u')$. 
  Now if $y\in \seg{uq}$
   with $d(y,u) > t_0 + 2D$, let $y'\in r'$ be such
  that $d(y,y') \le D$. Note that such a $y'$ exists from the CAT(0) comparison for
  the triangle $\triangle uu'q$. Then $d(y',u') > t_0$, so
  $d(y',r) > b/2 + D$. We conclude that $d(y,r)>b/2$. Hence the left
  endpoint of $J$ must be at most distance $t_0+2D$ apart from $u$, which means
  the left endpoint of $\pi_r(J)$ is at most distance $t_0+2D+b/2$ apart from $u$.

  Next, we prove that the length of $\pi_r(J)$ is bounded.  Let $d\in (0,b/2)$
  be small enough (say less than $b/4$) and let $K_0>0$
  be the constant from Lemma \ref{lem:I lower bound} corresponding to $e=b/2,\bar\ep$ and $d$. 
  Let $T>0$
  be such that $r([T,\infty))$ is contained in $U_{d,\bar\ep}(\omega)$,
    which is possible by the hypothesis that $r$ is strongly asymptotic to the $\bar\ep$--thick part of $\calS(\omega)$. 

 If   $\pi_r(J)$ has length greater than $n+1+T$, then there exist $n$
 disjoint intervals of length $1$, $I_i,\; i=1,\ldots,n$ 
 in the interior of $\pi_r(J)$ and in $U_{d,\bar\ep}(\omega)$. For $i=1,\ldots,n$ let
  $J_i$ be the subinterval of $J$ whose $\pi_r$-image is the interval
  $I_i$. We may choose the $I_i$ so that each $J_i$ is disjoint from
  $q$ (possibly discarding one if necessary). 
  Then $Q$ contains the regions $Q[J_i;r]$, and each of these
  contains a subrectangle 
  $Q_0^{b/2}[J_i;r]$. From Lemma \ref{lem:I lower bound} then we have
  $$
  \calI(Q_0^{b/2}[J_i;r]) \ge K_0
  $$
  where $K_0$ depends only on $b$ and $\bar\ep$.

  Thus by the monotonicity of $\calI$ (\ref{I monotonic}),
  $\calI(Q) \ge nK_0$. However, $\calI(Q)$ is controlled by the 
  Gauss-Bonnet theorem (Theorem \ref{thm : GB2}), which in this
  case gives $\calI(Q) \le \pi$.  This bounds $n$, and hence the length of
  $\pi_r(J)$, giving Claim \ref{claim:compact J}. 

The right endpoint of $\pi_r(J)$, then is a bounded distance
from $u$ and at distance $b/2$ from $\seg{pq}$. This proves the
lemma. 
\end{proof}

\begin{proof}[Proof of Theorem \ref{thick bottleneck}]
The proof will reduce easily to the following statement: 
\begin{lem}\label{lem:twistbd}
There is a compact subset $K_1\subsetneq \Teich(S)$ so that, for points $p\in r,q\in r'$
sufficiently far out, the segment
$\seg{pq}$ intersects $K_1$. 
\end{lem}
\begin{proof} 
 The proof of Lemma \ref{lem:distance bound} actually gave us a point $\hat{z}$ on $\seg{pq}$ such that $d(\hat{z},r)\leq b/2$
 (the right end point of $J$). Then since $r$ is strongly asymptotic to $\calS_{\bar\ep}(\omega)$, moving $\hat{z}$ along $\seg{pq}$ toward $p$ a bounded distance we obtain a
point $z\in \seg{pq}$, and the point $y = \pi_r(z)$ on $r$, 
such that:

\begin{enumerate}
  \item $d(y,\calS_{\bar\ep}(\omega)) < b/4$
  \item $d(z,y) \le b/2$
  \item $d(z,\hat{z}) > 1$ 
    \item $d(z,u)$ is bounded independently of $p$ and $q$.
\end{enumerate}

We can similarly move $\hat{z}$ toward $q$ to obtain $z'$
and $y'=\pi_{r'}(z')$ so that the inequalities (1)-(4) hold  for the points $u',y',z'$ and the multicurve $\omega'$.
Thus the geodesic hexagon $H$ joining the vertices
\[ u', u, y, z, z', y' \]
in cyclic order
has bounded total length independently of $p$ and $q$  (Figure \ref{hexagon}). 

\realfig{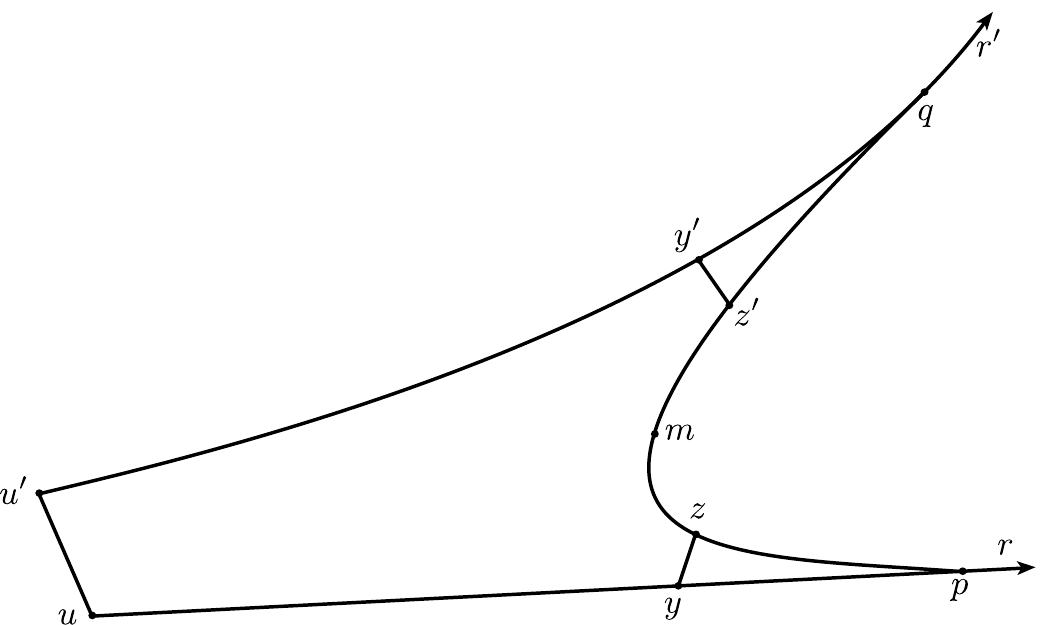}{The bounded diameter hexagon $H$ determined by $u',u,y,z,z',y'$.}
\label{fig:hexagon}

We will next show that
\begin{itemize}
  \item[(*)] $d_\alpha(z,z')$ is bounded above for each curve $\alpha$, by a quantity
    independent of $\alpha$, $p$ and $q$, and 
\item[(**)] $\ell_\alpha$ is bounded below on $\seg{zz'}$ by a positive constant
independent of $\alpha$, $p$ and $q$.
\end{itemize}

Fix any curve $\alpha$. 
Since the multicurves $\omega$ and $\omega'$ fill
$S$, $\alpha$ must intersect at least one of them. Suppose that $\alpha\pitchfork \omega$.
Then we have the following uniform bounds
(all independent of $p$, $q$ and $\alpha$) for the $\alpha$ coefficients:

\begin{enumerate}[(i)]
\item  $d_\alpha(u,u')$ is uniformly bounded; because the points $u,u'$ are fixed.

\item $d_\alpha(u,y)$ is uniformly bounded; this is because $y$ varies along a compact
  interval in $r$.  To see that the bound is independent of $\alpha$, note that the set of
  Bers markings that can occur for values of $y$ in this interval is finite, and the set
  of $\alpha$ with $d_\alpha(u,\mu)$ large for any given marking $\mu$ is finite, by an
  application of the distance formula (\ref{eq:dist}).

   \item  $d_\alpha(y,z)$ is uniformly bounded; by inequalities (1) and (2), we have that
    $\seg{yz}$ is in $U_{b,\bar\ep}(\omega)$.  Within this neighborhood there is an upper
     bound on the length of $\omega$, and since $\alpha\pitchfork\omega$ this means that
     $\pi_\alpha(\seg{yz})$ is uniformly close to $\pi_\alpha(\omega)$.

   \end{enumerate}
   
   Now further suppose
that $\alpha\notin\omega'$. Then we have the bounds:
   
   \begin{enumerate}[(a)]
  \item $d_\alpha(u',y')$ is uniformly bounded; this follows just as in (ii) above.
\item $d_\alpha(y',z')$ is uniformly bounded; since $\seg{y'z'}$ is in $U_{b,\bar\ep}(\omega')$ and $\alpha\notin \omega'$
 the bound may be obtained similarly to (iii). Here, while
 $\alpha$ may not intersect $\omega'$, it does intersect curves in $S\ssm\omega'$ whose lengths
 can only vary by a bounded amount in $\seg{y'z'}$ because $S\ssm\omega'$ is uniformly thick along $\seg{y'z'}$ by Lemma \ref{lem : b-nbhd}. 
  \end{enumerate}

Combining  (i)-(iii) and (a) and (b), we obtain a uniform bound on 
$d_\alpha(z,z')$ for all curves $\alpha$ such that $\alpha\pitchfork\omega$
and $\alpha\notin\omega'$,  and all $p\in r$ and $q\in r'$. 

Now since $z\in U_{\bar\ep,b}(\omega)$ and $z'\in U_{\bar\ep,b}(\omega')$,
there is a positive lower bound for the length of any $\alpha$ as above at the points $z$ and $z'$
by Lemma \ref{lem : b-nbhd}. Combining this with
the bound for $d_\alpha(z,z')$ and  appealing to Theorem \ref{thm : shtw2}, we obtain a uniform positive lower bound for the length of 
$\alpha$ along $\seg{zz'}$.

\medskip

Now suppose that $\alpha\pitchfork \omega$ and $\alpha\in \omega'$. First we show
 that $\ell_\alpha$ is uniformly bounded below on
$\seg{zz'}$. If not, we can fix $\alpha$ among the finitely many choices, and obtain
a sequence of segments $\{\seg{p_nq_n}\}_n$ with subsegments 
$\seg{z_nz'_n}$, and $x_n\in \seg{z_nz'_n}$ such that $\ell_\alpha(x_n)\to 0$ as $n\to\infty$.
We may assume that $q_n$ goes to $\infty$ along $r'$. 

Now since $q_n$ goes to infinity along $r'$, and $r'$ is strongly asymptotic to
$\calS_{\bar\ep}(\omega')$, we have $\ell_\alpha(q_n)\to 0$ as $n\to\infty$. Convexity of length-functions then
implies that the length of $\alpha$ goes to $0$ uniformly along $\seg{x_nq_n}$.
Thus we may find subintervals $\eta_n$ of $\seg{p_nq_n}$ of fixed length, centered at $x_n$ so that
the length of $\alpha$ goes to 0 on one side of $x_n$ and is bounded away from 0 on the opposite
endpoint of $\eta_n$. This contradicts Lemma \ref{lem:q-nonref}. Thus we obtain the desired lower bound for the length of $\alpha$.

Now we show that $d_\alpha(z,z')$ is uniformly bounded.  If not, 
again we can fix $\alpha$ among the finitely many choices, and obtain
a sequence of segments $\{\seg{p_nq_n}\}_n$ with subsegments 
$\seg{z_nz'_n}$ such that $d_\alpha(z_n,z'_n)\to \infty$ as $n\to\infty$. 

Since $z_n\in U_{b,\bar\ep}(\omega)$ and $\alpha\pitchfork \omega$, again by Lemma \ref{lem : b-nbhd}
$\ell_\alpha(z_n)$ is bounded below by a positive constant for all $n$. Moreover, the lengths of $\seg{z_nz'_n}$ are
bounded below by (3) and bounded above by (4) independently of $n$.  
Thus we may apply Theorem \ref{thm : twsh} to see that
after possibly passing to a subsequence there is a point $x_n\in \seg{z_nz'_n}$ such that $\ell_\alpha(x_n)\to 0$ as $n\to\infty$.
But then as we saw above this leads to a contradiction, 
showing that $d_\alpha(z,z')$ must be uniformly bounded above.

 Thus we obtain both (*) and (**) for all $\alpha\pitchfork\omega$.
The case of $\alpha\pitchfork \omega'$ proceeds similarly, exchanging the roles of $\omega$ and
$\omega'$.

With  (**) established, we conclude that $\seg{zz'}$ lies in the $\varepsilon$-thick part
of $\Teich(S)$, for $\varepsilon>0$ independent of $p$ and $q$. Since there is an upper
bound on the ratio of Teichm\"uller norm over WP norm in the $\varepsilon$-thick part, and 
the WP length of $\seg{zz'}$ is bounded
above by (4), we conclude that $\seg{zz'}$ has bounded Teichm\"uller length, and hence
$d_\alpha(m,z)$ and $d_\alpha(m,z')$ are uniformly bounded above for any $m\in \seg{zz'}$.

Now when $\alpha\pitchfork\omega$ we extract an upper bound for $d_\alpha(u,z)$ from (ii-iii), and when 
$\alpha\pitchfork\omega'$ we similarly extract an upper bound for
$d_\alpha(u,z')$ from the analogues of (i-iii). Putting these bounds together we obtain a uniform upper bound for
$d_\alpha(u,m)$ for all $m\in \seg{zz'}$.

We claim now that this implies that $m$ (and hence all of $\seg{zz'}$) must remain in some
compact subset $K_1$ of $\Teich(S)$. For we have an upper bound on $d(u,m)$ by the bounds
on the hexagon $H$, and  by Lemma \ref{thm : shtw2}, 
a sequence of segments $\seg{um}$ (for varying $p,q$) degenerates only by
producing arbitrary large twistings about a collection of curves, which is impossible by the 
bounds we just established on all $d_\alpha(u,m)$. Thus $\seg{um}$ has a uniformly bounded
length and remains in some thick part of $\Teich(S)$
(independently of $p,q$), so again has uniformly bounded Teichm\"uller length. 
But balls in the \T metric are compact and this gives us $K_1$. 
\end{proof}

From the lemma we obtain
points $\hat{p}\in r,\hat{q}\in
r'$ so that  
 all geodesics with end points $p,q$ on $r,r'$, further out than $\hat{p},\hat{q}$
 respectively,  intersect
 $K_1$. Now  letting 
 $K$ be the union of $K_1$ and the segments
 $\seg{u\hat{p}}$ and $\seg{u'\hat{q}}$, we have the desired bottleneck.
 This concludes the proof in the case that $r,r' \subset \Teich(S)$.

If $r$ is in $\calS_{\bar\ep}(\omega)$, using Proposition \ref{prop : asymp-strongasymp}
we can find $r_1$ in $\Teich(S)$ strongly asymptotic to it,
and similarly $r'_1$ strongly asymptotic to $r'$. The interior version of the theorem
gives a compact bottleneck $K$ for $r_1,r'_1$. There exists $\delta_1>0$ so that the
$\delta_1$-neighborhood of $K$ is still contained in a compact set $K_2\subset\Teich(S)$. Now choose $T$ sufficiently large
that $d(r(t),r_1(t))$ and $d(r'(t),r'_1(t))$ are less than $\delta_1$ for all $t>T$. Thus
any geodesic $\gamma$ joining $r(t)$ and $r'(s)$ for $t,s>T$ lies within $\delta_1$ of
a geodesic passing through $K$, and hence passes through $K_2$.  
 \end{proof}

 \subsection{Visibility} \label{subsec:visibility}

In this subsection we apply the Bottleneck theorem (Theorem \ref{thick bottleneck})
to show that any two geodesic rays that are strongly asymptotic to the $\bar\ep$--thick parts of two strata
determined by two filling co-large multicurves have the visibility property.

\begin{thm}\label{thm:asymplargevis}\tn{(Asymptotic Large Visibility)}
Suppose $r,r'$ are two infinite rays in $\oT{S}$ that are strongly asymptotic to, or
contained in, $\calS_{\bar\ep}(\omega)$
and $\calS_{\bar\ep}(\omega')$, respectively,  where $\omega$ and $\omega'$ are co-large
multicurves that fill $S$. 
Then  there exists a biinfinite geodesic $g\subsetneq \Teich(S)$ which is strongly asymptotic to $r$ in forward time 
and is strongly asymptotic to $r'$ in backward time.
\end{thm}
\begin{proof}  
The argument is essentially the same as in Theorem 1.3 of \cite{bmm1}, which obtains the
visibility property when $r$ and $r'$ are recurrent. We sketch here the mild changes
needed in our setting. 

Let $g_n$ denote the geodesic segment $\seg{r(n)r'(n)}$.
By Theorem \ref{thick bottleneck} there is a compact set $K\subsetneq\ Teich(S)$
so that $g_n\cap K\neq\emptyset$.

Let $v_n$ be a point of $g_n\cap K$. After possibly passing to a
subsequence the points $v_n$ converge to a point $v\in K$.

Reparametrize $g_n$ so that $g_n(0) = v_n$. 
To extract a limit of the $g_n$ we use the CAT(0) geometry of $\oT{S}$,
via Lemma 8.3 in \cite[\S II.8]{bhnpc}, together with the fact that $g_n(0)\to v$, to show that for each $t\in\RR$ 
the sequence $g_n(t)$ (defined
for large $n$) is a Cauchy sequence. We thus obtain a limiting geodesic $g$ in $\oT{S}$,
which is asymptotic to $r$ in forward time and to $r'$ in backward time.

As in \cite{bmm1} and Proposition \ref{prop : asymp-strongasymp}, we use the non-refraction property (Theorem \ref{thm:nonref}) to argue
that $g$ is in fact contained in $\Teich(S)$.

Finally, Proposition \ref{prop : asymp-strongasymp} implies that $g$ is {\em strongly}
asymptotic to $r$ in forward time and $r'$ in backward time. 
\end{proof}

\section{Indirect shortening along geodesic segments}\label{sec:example}

This section is devoted to the proof of Theorem \ref{thm:WP mismatch}, which follows very
directly  from  Theorem \ref{thm:one-step filling} below. We begin with
a discussion of the phenomenon of indirect shortening. 

\subsection{Indirect curve shortening}\label{subsec:indirect}

In the setting of \T geodesics the connection between short
curves and large subsurface coefficients was explored by Rafi in
\cite{rshteich}. He showed that given $\ep>0$ there exists $A\geq 1$ so that, 
if $g:I\to \Teich(S)$ is a \T geodesic with end invariant
$(\nu^+,\nu^-)$, then for any subsurface $Z\subseteq S$ we have
\[
d_Z(\nu^+,\nu^-) \geq A \implies \inf_{t}\ell_{\partial Z}(g(t)) < \ep. 
\]
The natural converse statement, which is motivated by the situation for
Kleinian surface groups (see Theorem \ref{thm:minsky}), would be
that given $A\geq 1$ there exists $\ep>0$ such that, 
if a curve $\gamma$ satisfies
\[
\inf_{t}\ell_\gamma(g(t)) < \ep
\]
then there exists a subsurface $Z$ with $\gamma\subseteq \partial Z$, and
\[
d_Z(\nu^+,\nu^-) \geq A. 
\]
This converse, however, does not hold in general. Rafi in the proof of Theorem 1.7 in \cite{rshteich} gave and analyzed
examples of sequences of geodesic segments $g_n$ with end invariants $(\nu^+_n,\nu^-_n)$ 
and a curve $\gamma$ for which
$\inf_{t}\ell_\gamma(g_n(t))$ is arbitrarily small, while $d_Z(\nu^+_n,\nu^-_n)$
remains bounded for all $Z\subseteq S$ with $\gamma\subseteq \partial Z$. 

Rafi's examples exhibited a somewhat more complex feature we might
call {\em indirect shortening}. To present the example we start with two definitions:

\begin{definition}
Given $A\geq 1$, a subsurface $Z\subseteq S$ and a pair of markings or laminations
$(\mu,\mu')$ we define
\[
\calL_A(Z,\mu,\mu') := \{ X\subseteq Z: d_X(\mu,\mu')> A \}.
\]
We define $\calL^{\na}_A(Z,\mu,\mu')$ as the subset of non-annular
surfaces in $\calL_A(Z,\mu,\mu')$. 
\end{definition}

\begin{definition}(Filling a subsurface)
Let $Z\subseteq S$, we say that a collection $\calZ$ of subsurfaces of $Z$ {\em fills} $Z$ if
 any curve $\alpha\in\calC(Z)$ intersects at least one $X\in\calZ$ essentially.
\end{definition}

In fact, Rafi in \cite[\S 6]{rshteich} gave examples of geodesics $g$ with the following
property: 

\begin{definition}\label{def:indirshort}(Indirect curve shortening)
Given $\ep>0$ and $A\geq 1$,
$\inf_{t}\ell_\gamma(g(t)) < \ep$ and there exists a subsurface $Z$ with $\gamma\subseteq
\partial Z$ that $\calL_A(Z,\nu^+,\nu^-)$
fills $Z$, but $Z\notin \calL_A(Z,\nu^+,\nu^-)$.
\end{definition}
We call the property {\em indirect curve shortening} because the subsurface $Z$ itself
does not have a big projection coefficient.

More recently it has become clear that this condition, also, does not
hold for \T geodesics in general; see \cite{shteich-revised}. Nevertheless let us state
as a conjecture the following characterization for short curves of \W geodesics.

 \begin{conjecture}\label{conj : coeff length}
  For any $\ep>0$, there exists $A(\ep)\geq 1$, and for any $A\geq 1$ there exists
  $\ep(A)>0$, such that the following holds. 
  Let $g:I\to \Teich(S)$ be a WP geodesic with end invariant $(\nu^+,\nu^-)$, 
  \begin{enumerate}
\item If $Z$ is a non-annular subsurface for which 
    $\calL^{\na}_{A(\ep)}(Z,\nu^-,\nu^+)$ fills $Z$
    or $Z$ is an annulus with $d_Z(\nu^+,\nu^-)>A(\ep)$, then 
 $\inf_{t} \ell_{\partial Z}(g(t)) <  \ep$. 
\item If $\inf_{t} \ell_\alpha(g(t)) <  \ep(A)$ for a curve $\alpha$, 
then either there exists a subsurface $Z$ with $\alpha\in\bd Z$ such that
$\calL^{\na}_{A}(Z,\nu^+,\nu^-)$ fills $Z$, or $d_\alpha(\nu^+,\nu^-)>A$. 
\end{enumerate}
 \end{conjecture}

 \subsection{The basic example}\label{subsec:basic}
To set the stage for our example of WP geodesics with indirect curve shortening property,
 consider the configuration of subsurfaces of $S$:
 \begin{equation}\label{fig:oneZ}
  \begin{tikzcd}[row sep=tiny,column sep=tiny]
 &    S   &  \\ 
 &    Z \arrow[u,hook,no head] & \\
    X_1 \arrow[ur,hook,no head] & &  X_2  \arrow[ul,hook',no head]
  \end{tikzcd}
\end{equation}
where we assume $Z$ is large in $S$, and $X_j$ is large
in $Z$ for $j=1,2$. We moreover assume that $\bd X_1$ and $\bd X_2$ fill $Z$
and that no boundary curve of $X_j$ is a boundary curve of $Z$ for $j=1,2$. We call this
a {\em one-step large filling configuration}; see Figure \ref{fillingarcs} for an example.

\realfig{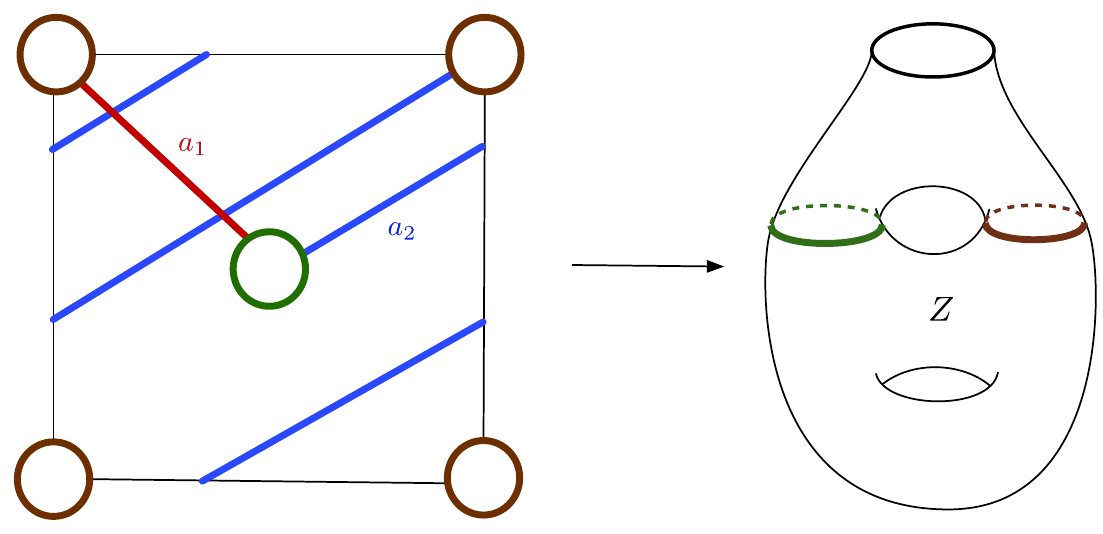}
{One-step large filling: The subsurface $Z$ is a two-holed torus. In its abelian cover on the left we
  indicate two arcs $a_1,a_2$ connecting the components of $\bd Z$ which fill $Z$. A
  regular neighborhood of $a_i \union \bd Z$ is a three-holed sphere $U_i$, and we define
  $X_i = Z\ssm U_i$. Note that $\bd X_i$ and $\bd Z$ have no components in common. 
On the right we indicate $Z$ as a subset of $S$. }

We can now state the main theorem of Section \ref{sec:example}. The proof will follow over
the next few subsections.
\begin{thm}\label{thm:one-step filling}
There exist $A_1\geq 1,\ep_0>0$ such that for each $\epsilon >0$ and a one-step
large filling configuration $Z, X_1, X_2$, there is
a \W geodesic segment $\overline{pq}$ in $\Teich(S)$ such
that
\begin{itemize}
  \item $\calL^{\na}_{A_1}(S,p,q)=\{X_1,X_2\}$, and $d_\gamma(p,q)\leq 1$ for all $\gamma\subseteq\bd Z$, 
    \item The endpoints $p,q$ are in the $\ep_0$--thick part of
      $\Teich(S)$, and
      \item $\inf_{x\in \overline{pq}} \ell_{\bd Z}(x) < \ep$.
\end{itemize}
\end{thm}

\subsection{Using visibility in a stratum}
Fix for the rest of the section a one-step large filling configuration $Z,X_1,X_2$ in $S$.

Since $Z$ is large the stratum $\calS(\bd Z)$ can be identified with $\Teich(Z)$ after replacing boundary curves with punctures,
and similarly the strata of $\bd X_i$ within this stratum, namely $\calS(\bd
X_i)\intersect\calS(\bd Z)$, can be identified with $\Teich(X_i)$ for $i=1,2$. 
Let $f_1$ and $f_2$ be partial pseudo-Anosov maps supported in $X_1$ and $X_2$
respectively, so that their axes $g_1$ and $g_2$ lie in the $\ep$-thick part of $\Teich(X_i)$ for
some $\bar\ep>0$. Fix rays $g_2^+$ and $g_1^-$ on these axes. 

We can apply Theorem \ref{thm:asymplargevis}, with $Z$ playing the role of $S$ in that theorem, to obtain a
biinfinite geodesic $h$ in $\Teich(Z)$ asymptotic to the rays $g_1^-$ in backward time and
$g_2^+$ in forward time.

The  geodesic segment examples for Theorem \ref{thm:one-step filling} will be obtained from $h$,
viewed in $\calS(\bd Z)$, by pushing points far out along $h$ slightly away from the
stratum. The key will then be to show that, for these geodesics, the subsurface
coefficients behave as required, and the length of $\bd Z$ becomes very small near the
center.

\subsection{Controlling subsurface coefficients along $h$}
Recall that $g_1$ and $g_2$ are geodesics in the $\bar\ep$--thick parts of $\Teich(X_1)$ and $\Teich(X_2)$, respectively,
and let $b>0$ be the corresponding constant to $\bar\ep$ from Lemma \ref{lem : b-nbhd}. 

Let us fix a parametrization of the bi-infinite geodesic $h:\RR\to \Teich(Z)$ and of the
geodesics $g_1, g_2$ by arclength, so that $d(h(t),g_2(t)) \to 0$ and
$d(h(-t),g_1(-t))\to 0$ as $t\to\infty$. We may also assume that $h(0)$ is at least distance $b$
away from the strata containing $g_1$ and $g_2$.

Let $s_1<0$ be so that $h|_{(-\infty,s_1]}$ and $g_1|_{(-\infty,s_1]}$ are the largest portions of $h$ and $g_1$ that $b$--fellow travel,
 and let $s_2>0$ be so that $h|_{[s_2,\infty)}$ and $g_2|_{[s_2,\infty)}$
 are the largest portions of $h$ and $g_2$ that  $b$--fellow travel.

\begin{lem}\label{lem:h subsurf}
 There is an $A\geq 1$, so that for any $t,t'\in (-\infty,s_1]$ we have 
   \begin{itemize}
   \item $\calL_A(S,h(t),h(t'))\subseteq \{X_1,\gamma\in \bd X_1\}$,
   \item $d_{X_1}(h(t),h(t'))\asymp d(h(t),h(t'))$.
   \end{itemize}
  where the constants of the coarse inequality depend only on $f_1$. 
  A similar statement holds for  $t,t'\in [s_2,\infty)$ the subsurface $X_2$ and $f_2$. 
  
  Moreover, for any $t,t'\in \RR$ we have
 \[\calL_A(S,h(t),h(t'))\subseteq \{X_1,X_2, \gamma\in \bd X_1\cup\bd X_2\}.\]
\end{lem}
\begin{proof}
When $t,t'\in (-\infty,s_1]$ the points $h(t)$ and $h(t')$ are within distance $b$ of points $y_t$ and $y_{t'}$ on $g_1$.
Then by the coarse Lipschitz property of subsurface projections (Lemma \ref{lem:dWP-dY})
we have
\begin{equation}\label{eq:dYhg}
\Big|d_Y(h(t),h(t'))- d_Y(y_t,y_{t'})\Big|\leq D
\end{equation}
for some $D>0$ and all non-annular subsurface $Y\subseteq S$. Moreover, the segments $\seg{h(t)y_t}$ and $\seg{h(t'))y_{t'}}$
are in the $b$--\nbhd of $g_1$ which is in $\calS_\ep(\bd X_1)$, thus by Lemma \ref{lem : b-nbhd} the segments are 
away from all strata except $\calS(\sigma)$ with $\sigma\subseteq \bd X_1$.
Now let $\gamma$ be a curve which is not in $\bd X_1$, 
then we have a uniform lower bound for the length of $\gamma$ along $\seg{h(t)y_t}$ and  $\seg{h(t')y_{t'}}$.
Applying Theorem \ref{thm : twsh} then we obtain a uniform upper bound for $d_\gamma(h(t),y_t)$ and  $d_\gamma(h(t),y_{t'})$.
 This implies that (\ref{eq:dYhg}) also holds for 
 for all annuli whose core curves are not in $\bd X_1$.

Now note that $y_t$ and $y_{t'}$ are on an axis of $f_1$, so by Lemma \ref{lem:dWP-dX}, $d_Y(y_t,y_{t'})$
is uniformly bounded for all subsurfaces $Y$ except $X_1$ and the annuli with core curves in $\bd X_1$.
Thus by (\ref{eq:dYhg}) $d_Y(h(t),h(t'))$ is uniformly bounded for all subsurfaces $Y$ except $X_1$ and the annuli with core curves in $\bd X_1$.
This is the first bullet of the lemma.

Moreover, note that by Lemma \ref{lem:dWP-dX}, 
\[d_{X_1}(y_t,y_{t'})\asymp d(y_t,y_{t'})\]
so by (\ref{eq:dYhg}) we obtain the second bullet of the lemma.
When $t,t'\in [s_2,\infty)$ the bullets are proved similarly where $X_1$ is replaced by $X_2$.

To see the second part of the lemma, note that the segment $\seg{h(s_1)h(s_2)}$ is fixed, so 
 there is a $D_1\geq 1$ that bounds all projection coefficients of any pair of points on $\seg{h(s_1)h(s_2)}$.
  Combining this bound and the bounds from the first part of the lemma with the triangle inequality for each non-annular subsurface $Y$ 
  which is not $X_1,X_2$ or an annulus
 with core curve a boundary component of $X_1$ or $X_2$ 
 we find that the projection coefficient of the subsurface is uniformly bounded 
 giving us the second part of the lemma. 
\end{proof}

\subsection{Controlling the length of $\bd Z$}
The following theorem is the main ingredient of the proof of Theorem \ref{thm:one-step
  filling}. It says roughly that,
in a geodesic fellow-traveling a sufficiently long part of our geodesic $h$, if
the length of $\bd Z$ is bounded at the endpoints then it
becomes very short near the center. 

\begin{thm}\label{thm : bdZ short}
  Let $h$ be the geodesic constructed above using a one-step large filling configuration
  $Z,X_1,X_2$. 
  Let $D\geq 1$ and let $I_n=[a_n,b_n]$ be a sequence of intervals so that  $0\in I_n$ and $I_n\subseteq I_{n+1}$ for $n\in \NN$,
 and $\cup_n I_n=\RR$.
   Let \{$\zeta_n:I_n\to \Teich(S)$\} be a sequence of WP geodesic segments 
such that $\zeta_n$ and $h|_{I_n}$ are $D$-fellow travelers as parameterized geodesics. 
  Moreover, suppose that the length of $\bd Z$ is bounded above at the end points of $\zeta_n$
  independently of $n$.
  Then, there is a compact interval $J$ so that after possibly passing to a subsequence
\[\ell_{\bd Z}\circ\zeta_n \to 0\]
 uniformly on $J$.
 \end{thm}

  We wish to apply Theorem \ref{thm:lenapproach0} to the sequence of geodesics $\zeta_n:I_n\to \Teich(S)$ to prove the theorem. 
  By the hypothesis of the theorem and convexity of length-functions, $\ell_{\bd Z}\circ\zeta_n$ is uniformly bounded above on the intervals $I_n$,
  so we only require to show that there is an interval $J$ over which the lengths of all curves which are not a component curve of $\bd Z$
  are uniformly bounded below. More precisely,
  
  \begin{lem}\label {lem : J}
There is an interval $J\subseteq \RR$ and $c_1>0$ such that, 
for all curves $\gamma$ that are not components of $\bd Z$,
there is a lower bound 
\begin{equation}\label{length beta lower}
  \ell_\gamma\circ\zeta_n > c_1
\end{equation}
on $J$ 
for all $n$.
\end{lem}

The proof of Lemma \ref{lem : J} requires two lemmas.  First we obtain a lower bound for
lengths of most curves along $\zeta_n$:
 
\begin{lem}\label{lem : lb on I}
There exist $j,\ep>0$ and $A\geq 1$ such that, letting $I_n^- := [a_n+j, -j]$ and  $I_n^+ := [j,b_n-j]$,
for any curve $\gamma$ which is not in $\bd Z$ and intersects $X_1$ we have
the length lower bound
\begin{equation}\label{length lower bound}
\ell_\gamma(\zeta_n(t)) > \ep
\end{equation}
for all $t\in I_n^-$.
Similarly if $\gamma\pitchfork X_2$ then (\ref{length lower bound}) holds
when $t\in I^+_n$. 
\end{lem}

\begin{proof}
The idea is that, because $\zeta_n$ in the interval $[a_n,0]$ is roughly controlled by the geodesic $g_1$,
the only subsurface projections that can build up along $\zeta_n$ are in $\calC(X_1)$ (Lemma
\ref{lem:h subsurf}), but on the other hand short curves that appear in this interval must
give rise to large twists, using Theorem \ref{thm : shtw2}. 

First note that by Lemma \ref{lem:h subsurf}, 
\[d(h(t),h(t'))\asymp d_{X_1}(h(t),h(t')),\]
for all $t,t'\in(-\infty,s_1]$.
 Moreover, note that by (\ref{eq:brock}) $\rho:=Q\circ h:(-\infty,s_1]\to \calP(S)$ is a quasi-geodesic in $\calP(S)$ with quantifiers depending only on the topological type of $S$.
Also $\rho|_{I_n}$ and $Q\circ \zeta_n|_{I_n}$,
$D'=K_{\WP}D+C_{\WP}$  fellow travel in $\calP(S)$ as parametrized quasi-geodesics, where $K_{\WP}$ and $C_{\WP}$ are the constants in (\ref{eq:brock}). Then Theorem \ref{thm:anncoeffcomp} applied 
to $\rho$, the part of $\zeta_n$ that $D'$--fellow travels $\rho$ and the subsurface $X_1$ 
gives us constants $B,w>0$ and $\omega>0$ as follows: Let $\gamma$ be a curve such that $\gamma\pitchfork X_1$ and $\ell_\gamma(\zeta_n(t))<L_S$, 
so that $\gamma$ is in a Bers pants decomposition $Q(\zeta_n(t))$. Let $t_1=t-w$ and $t_2=t+w$ then

\begin{equation}\label{eq:anncoeffbd}
d_\gamma(h(t_1),\zeta_n(t_1))\leq B \;\;\text{and}\;\; d_\gamma(h(t_2),\zeta_n(t_2))\leq B
\end{equation}
thus
\begin{equation}\label{eq:anncoeffcomp}
\Big|d_\gamma(h(t_1),h(t_2))- d_\gamma(\zeta_n(t_1),\zeta_n(t_2))\Big|\leq 2B.
\end{equation}
Moreover,
\begin{equation}\label{eq : len bd}
\min\{\ell_\gamma(\zeta_n(t_1)),\ell_\gamma(\zeta_n(t_2))\}\geq \omega.
\end{equation}
Now let $j>|s_1|+w$ and let $n\in \NN$ be large enough and $t\in I^-_n$.

Suppose that 
$\ell_\gamma(\zeta_n(t))<\ep$ for an $\ep<\min\{\omega,L_S\}$ and $t\in [t_1,t_2]$. Then,
 noting that
$|t_1-t_2|$ is bounded independently of $n$ and $t$, we can apply Theorem \ref{thm : shtw2}
to $\zeta_n|_{[t_1,t_2]}$ to conclude that there is a choice of $\ep$ that implies
$d_\gamma(\zeta_n(t_1),\zeta_n(t_2))>A+2B$.

But then by (\ref{eq:anncoeffcomp}), $d_\gamma(h(t_1),h(t_2))>A$,
which contradicts the upper bound for subsurface coefficients from Lemma \ref{lem:h subsurf}. 
The contradiction shows that in fact the above $\ep$ is the desired lower bound for the length of a curve $\gamma\pitchfork X_1$
 on the interval $I_n^-$. The lower bound for the length of a curve $\gamma\pitchfork X_2$ on the interval $I_n^+$
can be obtained similarly choosing $j>s_2+w$.
\end{proof}

Next we obtain upper length bounds along $\zeta_n$ for $\bd X_1$ and $\bd X_2$ over intervals $I_n^-$
and $I_n^+$, respectively: 
\begin{lem}\label{lem : ub on I}
There exists $c>0$ such that for any $n\in \NN$, 
\begin{equation}\label{boundary upper bound 1}
  \ell_{\bd X_1}(\zeta_n(t)) \le c
\end{equation}
for all  $t\in I_n^-$, and 
\begin{equation}\label{boundary upper bound 2}
  \ell_{\bd X_2}(\zeta_n(t)) \le c
\end{equation}
for all  $t\in I_n^+$. 
\end{lem}
  
\begin{proof}
Let $t\in I^-_n$ and let $\eta:[0,l]\to \oT{S}$ be the
geodesic segment connecting $h(t)$ to $\zeta_n(t)$. 
The idea of the proof is to first obtain a lower bound along $\eta$ for the length of every curve that intersects $\bd X_1$,
which is similar to the proof of the previous lemma. Then, by a compactness argument 
appealing to the Geodesic Limit theorem (Theorem \ref{thm : geodlimit}) we establish the desired upper bound for the length of  $\bd X_1$.

Let $\gamma\pitchfork \bd X_1$, and let $t\in I_n^-$ and $\eta$ be as above.
We have $\ell_\gamma(\zeta_n(t))>\ep$ where $\ep$ is the constant from
Lemma \ref{lem : lb on I} above. Moreover, 
we also have a lower bound $\ell_\gamma(h(t))>\ep'>0$ using the Collar Lemma (\cite[\S 4.1]{buser}) with the fact 
that the length of $\bd X_1$ is bounded above along $h((-\infty,s_1])$. 

To bound the length of $\gamma$ from below on $\eta$, we will first obtain a bound on
$d_\gamma(\eta(0),\eta(l)) = d_\gamma(h(t),\zeta_n(t))$.

Since $\gamma$ is in the pants decomposition $Q(\eta(u))$ which is at most $D'$ from
$Q(\zeta_n(t))$, we can use Theorem \ref{thm:anncoeffcomp}, just as in the proof of 
Lemma  \ref{lem : lb on I}, to find parameter $t_2 > t$ with $t_2-t$
bounded above, and a bound $B_1$ such that
\begin{equation}\label{dgamma1}
  d_\gamma(\zeta_n(t_2),h(t_2)) \le B_1. 
\end{equation}
(Recall this is done by moving forward along $h$ and $\zeta_n$ just enough to obtain
points so far from $\pi_{X_1}(\gamma)$ in $\calC(X_1)$ that the path from $Q\circ h$ to $Q\circ\zeta_n$
passes only through curves transverse to $\gamma$). 

Next, we obtain
\begin{equation}\label{dgamma2}
  d_\gamma(\zeta_n(t),\zeta_n(t_2)) \le B_2
\end{equation}
by recalling from Lemma  \ref{lem : lb on I} that $\ell_\gamma\circ\zeta_n > \ep$ on
$[t,t_2]$, and then applying Theorem \ref{thm : twsh}.

Finally, we get
\begin{equation}\label{dgamma3}
  d_\gamma(h(t),h(t_2)) \le B_3
\end{equation}
directly from Lemma \ref{lem:h subsurf}. 

Putting (\ref{dgamma1}), (\ref{dgamma2}) and (\ref{dgamma3}) together we obtain a bound on
$d_\gamma(h(t),\zeta_n(t))$. Now, using Theorem \ref{thm : shtw2}, this gives us a lower
bound
\begin{equation}\label{lgamma lower eta}
\ell_\gamma(\eta(u)) > \ep''>0
\end{equation}
for all $u\in [0,l]$, and all $\gamma \pitchfork \bd X_1$.

Now assume that there is a sequence of geodesic segments $\eta_n:[0,l_n]\to \oT{S}$ as
above, connecting $h(t_n)$ to $\zeta_n(t_n)$ (for $t_n\in I_n^-$), and  $u_n\in [0,l_n]$ and 
 $\alpha\in \bd X_1$ so that $\ell_\alpha(\eta_n(u_n))\to\infty$ as $n\to\infty$.

Let the piece-wise geodesic segment $\hat\eta$ be be obtained from $\{\eta_n\}$ as in Theorem \ref{thm : geodlimit} (GLT), 
and multicurves $\sigma_i,\; i=0,\ldots,k+1$ and $\hat\tau$ be from the theorem. Since $\eta_n(0)$ is in the $b$--\nbhd of the axis of $f_1$
we may choose $\psi_n$ in GLT3 to be a power of $f_1$, which since $f_1$ is supported in $X_1$ does not change the homotopy classes of curves in $\bd X_1$.
Moreover, the lower bound (\ref{lgamma lower eta}) over $\eta_n$ for the lengths of all curves that intersect $\bd X_1$ 
 shows that $\sigma_i$ is disjoint from $\bd X_1$ 
and hence $\varphi_{i,n}$ which is a composition of $\psi_n$ and Dehn twists about curves in $\sigma_j,\; j=0,\ldots,i$ does not change homotopy classes of curves in $\bd X_1$.

After possibly passing to a subsequence $u_n\to u^*$, so the fact that $\ell_{\alpha}(\eta_n(u_n))\to\infty$ as $n\to\infty$ and GLT3 imply that
$\ell_\alpha(\hat\eta(u^*))=\infty$.
This means that $\alpha$ intersects a pinched curve along $\hat\eta$ 
and hence a multicurve $\sigma_i$. But we just said that $\bd X_1$ and $\sigma_i$ are disjoint.
This contradiction shows that the lengths of curves $\alpha\in \bd X_1$ are uniformly bounded
along $\eta$ and in particular at the end point $\zeta_n(t)$, as was desired. This
concludes the proof of (\ref{boundary upper bound 1}). 
The proof of (\ref{boundary upper bound 2}) for the length of $\bd X_2$ proceeds similarly.
\end{proof}

With Lemmas \ref{lem : lb on I} and \ref{lem : ub on I} in hand we can complete the proof
of Lemma \ref{lem : J}.

\begin{proof}[Proof of Lemma \ref{lem : J}.]

Let $\gamma$ be any curve which is not in $\bd Z$. If $\gamma$ intersects $\bd Z$ we already have
a lower bound for the length of $\gamma$ everywhere on $\zeta_n$.
Since $Z$ is large, we are left with the case that $\gamma$ is in $Z$.

When $\gamma$
overlaps both $X_1$ and $X_2$, let $w>0$ be as in the proof of Lemma \ref{lem : lb on I}. Moreover let $t_1=-j-w$ and $t_2=j+w$
and observe that $t_1\in I^-_n$ and $t_2\in I^+_n$ where $I^{\pm}_n$ are the intervals from Lemma \ref{lem : lb on I}.  
Then by Lemma \ref{lem : lb on I} we have that
 $\ell_\gamma(\zeta_n(t_1))$ and $\ell_\gamma(\zeta_n(t_2))$
 are at least $\ep$.

Thus we may apply Theorem \ref{thm : shtw2} to conclude that there is an $\ep''<\min\{\ep,L_S\}$ so that if $\min_{t\in [t_1,t_2]}\ell_\gamma(\zeta(t))< \ep''$, then
$d_\gamma(\zeta_n(t_1),\zeta_n(t_2))> A+2B$. 
From (\ref{eq:anncoeffbd}) then we see that $d_\gamma(h(t_1),h(t_2))>A$. But this again contradicts the bound
for subsurface coefficients in Lemma \ref{lem:h subsurf}. The contradiction shows that $\ep''$ is a lower bound for the lengths of all curves 
that are disjoint from $\bd X_1$ and are inside $Z$ on $[t_1,t_2]$ and in particular on $[-j,j]\subseteq [t_1,t_2]$.

Now consider $\gamma$ inside $Z$ which does not overlap $X_1$. Then it must be a boundary component
of $X_1$ and must intersect $X_2$.

By Lemma \ref{lem : lb on I} we know that $\ell_\gamma(\zeta_n(j)) > \ep$, and by Lemma \ref{lem : ub on I},
$\ell_\gamma(\zeta_n(t)) \le c$ for all $t<-j$. Suppose now that there is a sequence $t_n\in  [-j,j]$ with $\ell_\gamma(\zeta_n(t_n)) \to 0$
 as $n\to\infty$.

 We may restrict to a subsequence such that $t_n \to t^*$.
Since $\ell_\gamma(\zeta_n(j))>\ep$ we know that $t^* \le j$.
Now since $\ell_\gamma\circ \zeta_n$ is convex and bounded on the intervals $I_n^-$ whose lengths go to
$\infty$, we conclude that 
$\ell_\gamma\circ\zeta_n(t)\to 0$ for all $t<t^*$.
We can therefore find a sequence of intervals $[t_n-a,t_n+a]$ with fixed $a>0$ such that
$\ell_\gamma\circ\zeta_n \to 0$ on $[t_n-a,t_n]$ while $\ell_\gamma(t_n+a)$ is bounded
away from 0. This contradicts Lemma \ref{lem:q-nonref}.

 The contradiction shows that there is a lower bound for the lengths of curves that are inside $Z$ and are disjoint from $\bd X_1$
 on $[-j,j]$ as well. Therefore, $J:=[-j,j]$ is the desired compact interval of the lemma.
\end{proof}

\begin{proof}[Proof of Theorem \ref{thm : bdZ short}]
Lemma \ref{lem : J} gives us an interval $J$ over which the length of every curve that does not intersect $\bd Z$
is bounded below. Moreover by the assumption of the theorem and convexity of length-functions the lengths of all curves 
in $\bd Z$ are bounded along $\zeta_n$. Thus the theorem follows from Theorem \ref{thm:lenapproach0}.
\end{proof}

\subsection{Completing the proof of Theorem \ref{thm:one-step filling}}
  
  \begin{proof}[Proof of Theorem \ref{thm:one-step filling}]
  Let $a_n\to -\infty$ and $b_n\to\infty$, and let $I_n=[a_n,b_n]$.
  Also let $p_n,q_n$ be two points in the $b$--neighborhoods of $h(a_n),h(b_n)$, 
  respectively, that have injectivity radii at least $\ep_b$; see Lemma \ref{lem : b-nbhd}.
  Moreover, applying Dehn twists about curves in $\bd Z$ we can assume that
  \[d_\gamma(p_n,q_n)\leq 1\]
  for all $\gamma\subseteq \bd Z$.

After a slight adjustment of parameters let
 \[\zeta_n:I_n\to \Teich(S)\]
  be a parameterization of the geodesic segment $\seg{p_nq_n}$ by arclength where $d(\zeta_n(0),h(0))\leq b$.
  
First, note that the points $p_n$ and $q_n$ are in the $b$--\nbhds of the points $h(a_n)$ and $h(b_n)$, respectively,
so by Lemma \ref{lem:dWP-dY},
$d_Y(p_n,h(a_n))$ and $d_Y(q_n,h(b_n))$
are uniformly bounded 
for all non-annular subsurfaces $Y\subseteq S$ and $n\in \NN$.

Now note that by the second part of Lemma \ref{lem:h subsurf} there is an $A\geq 1$
 so that $\calL^{\na}_{A}(S,h(a_n),h(b_n))\subseteq \{X_1,X_2\}$.
 Thus enlarging $A$ we obtain an $A_1\geq 1$ so that $\calL^{\na}_{A_1}(S,p_n,q_n)\subseteq \{X_1,X_2\}$.

Now we show that $X_1$ and $X_2$ are in fact in $\calL^{\na}_{A_1}(S,p_n,q_n)$ for $n$ large enough,
note that by the second bullet of Lemma \ref{lem:h subsurf} we have
\[d_{X_1}(h(s_1),h(a_n))\asymp d(h(s_1),h(a_n))\]
which implies that $d_{X_1}(h(s_1),h(a_n))$ is arbitrary large for $n$ large enough (because $ d(h(s_1),h(a_n))$ gets arbitrary large).
By the first bullet of Lemma \ref{lem:h subsurf} $d_{X_1}(h(s_2),h(b_n))$ is bounded independently of $n$.
Moreover, $d_{X_1}(h(s_1),h(s_2))$ is bounded since $h(s_1)$ and $h(s_2)$ are fixed. 
 The above bounds combined with the triangle inequality
show that $d_{X_1}(h(a_n),h(b_n))$ is larger than $A_1$ for all $n$ large enough. 

The fact that $d_{X_2}(h(a_n),h(b_n))$ is larger than $A_1$ for all $n$ large enough can be proved similarly.
 Thus the first bullet of theorem holds for $\seg{p_nq_n}$ and $n$ large enough.
 
 The second bullet of the theorem holds immediately for all $p_n,q_n$ by the choice of the points.

 Now note that the points $p_n$ and $q_n$ are in the $b$--\nbhd of $\calS(\bd Z)$ so by Lemma \ref{lem : b-nbhd}) we have an upper bound for the length of
  $\bd Z$ at $p_n$ and $q_n$ independently of $n$. Also, since $p_n,q_n$ are in the $b$--\nbhds of two points 
  on $h$, $\seg{p_nq_n}$ and $h|_{I_n}$ $b$--fellow travel. Thus,
Theorem \ref{thm : bdZ short} applies to $\seg{p_nq_n}$,
giving us $\inf_{x\in\seg{p_nq_n}}\ell_{\bd Z}(x)<\ep$
for all $n$ large enough. Thus the third bullet of the theorem also holds for $\seg{p_nq_n}$ and all $n$ large enough.

As we saw above all of the bullets of the theorem hold for $\seg{p_nq_n}$ when $n$ is large enough completing the
proof of the theorem.
 \end{proof}

\subsection{Completing the proof of Theorem \ref{thm:WP mismatch}}

Take a one-step filling
configuration $Z,X_1,X_2$ in $S$, let $\gamma $ be a component of $\bd Z$, and let $p,q$
be as constructed in Theorem \ref{thm:one-step filling}. Then the conditions of
Theorem \ref{thm:WP mismatch} are satisfied, where one detail to check carefully is
the second bound
$$
\sup\Big\{d_Y(p,q) \ | \ Y\subseteq S, \gamma\subseteq\bd Y\Big\} \le A.
$$
But according to Theorem \ref{thm:one-step filling} the only subsurfaces where $d_Y(p,q) > A$
are $Y=X_1$ and $X_2$, and by definition those subsurfaces cannot have $\gamma$ in their
boundaries. This concludes the proof.

\section{Indirect shortening along closed geodesics}\label{sec:closed}

In this section we construct examples of closed \W geodesics which satisfy 
the indirect curve shortening property in Definition \ref{def:indirshort}. 
We construct such geodesics by approximation of the segments constructed in Theorem
\ref{thm:one-step filling} with arcs of closed geodesics while controlling end invariants
and their subsurface coefficients. 

  \begin{thm}\label{thm:closedgeodesic}
There exists $A\geq 1$ such that for each $\epsilon >0$ there is
a \pA map $\Phi$ with stable/unstable laminations $(\nu^+,\nu^-)$ and axis $A_\Phi$,
and a subsurface $Z\subsetneq S$ that for each $\gamma$ in $\bd Z$ we have 
\begin{equation}\label{eq:dYbd}
\sup_{\substack{Y\subseteq S: \gamma\subseteq \bd Y}}d_Y(\nu^+,\nu^-)\leq A,
\end{equation}
but 
\[\inf_{x\in A_\Phi} \ell_{\bd Z}(x) < \ep.\]
\end{thm}

  In the proof we use the following notation:
   If $f$ is a pseudo-Anosov or a partial pseudo-Anosov
  supported in a subsurface let $\nu^+(f)$ and $\nu^-(f)$ be the stable and unstable
  laminations of $f$ (considered without their measures). Similarly if $G$ is a directed WP geodesic let $\nu^+(G)$ and
  $\nu^-(G)$ be the ending laminations of the forward and backward rays $G^+ := G|_{[0,\infty)}$
    and $G^- := G|_{(-\infty,0]}$. The axis $A_f$ of $f$ is always oriented so that $\nu^\pm(A_f)
  = \nu^\pm(f)$. 

The main idea is to  approximate the configuration of $\S$\ref{subsec:basic} and Theorem \ref{thm:one-step filling}
 by axes of pseudo-Anosov maps. For this we can use the density of closed WP geodesics \cite[Theorem 1.6]{bmm1}, but it
will take some care to do it while controlling the ending laminations and their projections to the various
subsurfaces of interest. 

Let $X_1, X_2$ and $Z$ be the subsurfaces from the one-step large filling configuration in $\S$\ref{subsec:basic}.
Let $f_1,f_2$ be the partial \pA maps supported on $X_1,X_2$ respectively, with (oriented)
axes $g_1, g_2$, respectively.
Let $h$ be the biinfinite geodesic constructed in Theorem \ref{thm:asymplargevis},
which is forward asymptotic to $g_2$ and backward asymptotic to $g_1$.

\subsection{Overall construction}

Fix an oriented axis $G$ of a \pA map which is an $\ep_0$--thick WP geodesic (a geodesic that is entirely in the $\ep_0$--thick part of \T space), 
and a point $x$ on $G$. Define a sequence
$(G_n,x_n)$ with $n\in \ZZ$ as follows: For $n\ge 0$, set
$$
G_n = f_2^n(G), \qquad x_n = f_2^n(x),
$$
and
$$
G_{-n} = f_1^{-n}(G), \qquad x_{-n} = f_1^{-n}(x).
$$
We construct our desired sequence of \pA maps $\Phi_n, n\in \NN$, in two steps (see
Figure \ref{Phi-n}). 

\medskip

{\bf Step 1: } Use the Recurrent Visibility Theorem  \cite[Theorem 1.3]{bmm1} to obtain an oriented
geodesic $B_n$ strongly asymptotic to $G_n^+$ in forward time and $G_{-n}^-$ in backward time.
Since strongly asymptotic rays have the same ending laminations (a consequence of the definition), we
see that $\nu^+(B_n) = \nu^+(G_n)$ and
$\nu^-(B_n) = \nu^-(G_{-n})$.

\medskip

{\bf Step 2:} The Closed orbit density theorem \cite[Theorem 1.6]{bmm1} implies that we can approximate $B_n$ as
closely as we like by axes of \pA mapping classes. We select such an approximation $A_n =
axis(\Phi_n)$ according to the criteria below.

\realfig{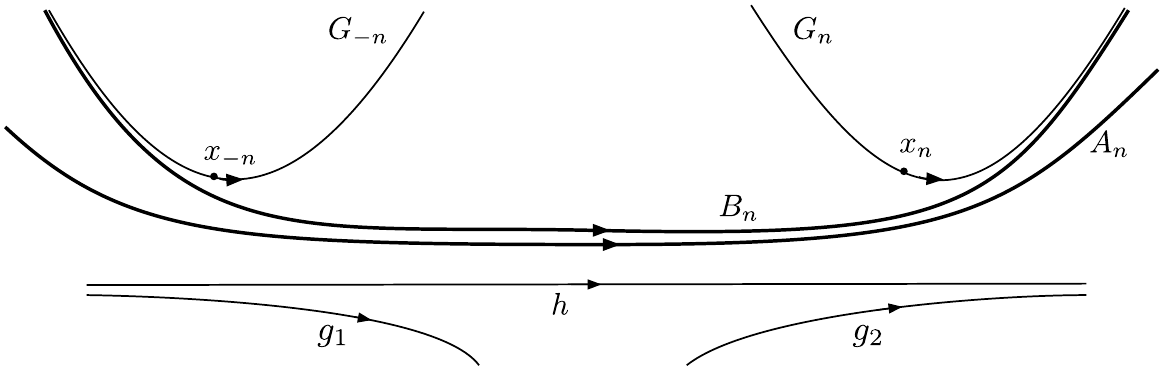}{The construction of $\Phi_n$ and its axis $A_n$.}

\medskip

The challenge will be to show that, with appropriate choices in Step 2, we obtain
axes $A_n$ which uniformly fellow-travel large segments of the geodesic $h$, which have
sufficient geometric control to enable us to apply Theorem \ref{thm : bdZ short} to show
that, in the middle of $A_n$ that fellow-travels $h$ there are points where the length of $\bd Z$ becomes arbitrarily
short, and to control the ending laminations $\nu^\pm(\Phi_n)$ sufficiently well to obtain
the bound (\ref{eq:dYbd}) on subsurface projections.

\subsection{Geometric control of $\{G_n\}_n$}

Recall that $G_n$ and $G_{-n}$ are $\ep_0$--thick so there is a $\delta>0$ 
so that the the $\delta$--\nbhds of the geodesics are disjoint from all completion strata
and there is a positive lower bound for all 
sectional curvatures in the $\delta$--\nbhds of the geodesics. 
By definition $G_n$ and $G_{-n}$  are strongly asymptotic to $B_n$ in forward and backward time,
respectively, but we will need uniform control, independent of $n$, on how quickly they
approach. This is the purpose of the following lemma:  
\begin{lem}\label{rectangle control}
There exists $D>0$  so that for each $n\in \NN$ large enough there is a point $y_n$ forward of
$x_n$ in $G_n$ such that
$$d(x_n,y_n) \le D \qquad \text{and} \qquad d(y_n, B_n) < \delta/2.$$
Similarly we have $y_{-n}$ behind $x_{-n}$ in $G_{-n}$ with
$$d(x_{-n},y_{-n}) \le D \qquad \text{and} \qquad  d(y_{-n}, B_n) < \delta/2.$$
\end{lem}

\begin{proof}
The proof uses the same ruled polygon technique as in \cite[\S4]{bmm1} \cite[\S6]{asympdiv} and Section
\ref{bottlenecks} of this paper. In preparation we first need the following estimate on the shape of the
configuration of $\{G_n\}_n$.

  \begin{lem}\label{CAT0 divergence}
    There exists an affine function $\varphi:\RR\to \RR$ with positive slope so that, for all $n\in \NN$ large enough, and
    any $z\in G_n$ 
    \[
    d(z,\seg{x_{-n}x_n}) \ge \varphi(d(z,x_n))
    \]
    and similarly $    d(z,\seg{x_{-n}x_n}) \ge \varphi(d(z,x_{-n}))$ for any $z\in G_{-n}$. 
  \end{lem}

  \begin{proof}
    Note first that $G$ and $g_2$ are not asymptotic since $G$ is a thick geodesic and $g_2$ is
    contained in a stratum.
    Since the WP metric is CAT(0), distances to geodesics are convex so there is an affine
    function $\varphi_0(t) = a_0t-c_0$ with $a_0>0$  so that for any $z\in G$ we have
    $$ d(z,g_2) \ge \varphi_0(d(z,x)).$$
    Now since $f_2$ preserves $g_2$,      for all $z\in G_n$ we have 
\begin{equation}\label{g2 growth}
  d(z,g_2) \ge \varphi_0(d(z,x_n)).
\end{equation}
We next want to prove a similar inequality for $h$ replacing
    $g_2$.

    Let $q$ be the nearest point to $x$ on $g_2$ and let $q_n = f_2^n(q)$. Then, we have
    $d(x_n,q_n) = d(x,q) \equiv d_0$.
 Moreover, since $h$ is asymptotic to $g_2$ in forward time, there is a sequence $\tau_n\to \infty$ so
 that the interval of radius $\tau_n$ in $g_2$ around $q_n$ is within distance $1$ of $h$
 for all $n$ large enough.

 Fix $d_1 > d_0+4$ and let $s\in \RR$ be such that $\varphi_0(s) = d_1$.
  We claim that, for large enough $n$ and for a point $z_n\in G_n$ with
 $d(z_n,x_n) = s$ we have that
 \begin{equation}\label{h increase}
   d(z_n,h) > d_1-2.
 \end{equation}
Suppose not, and choose $n$ so that $\tau_n \gg d_1 + d_0 + s$.
The nearest point to $z_n$ on
$h$ is then within the interval that 1-fellow-travels $g_2$, so we have that
$d(z_n,g_2) \le d_1 -2 +1  < d_1$;  but this contradicts  (\ref{g2
  growth}), and thus (\ref{h increase}) holds. Now note that the distance of $x_n\in G_n$ to $h$ is at most $d_0+1$
and the distance of $z_n\in G_n$ to $h$
  is between $d_1-2$ (by (\ref{h increase})) and $d_0+1+s$ (by the triangle inequality). 
For any positive convex function $f:[0,\infty)\to \RR$  with $f(0)< f(s)$ we have $f(t) \ge
  \frac{f(s)-f(0)}{s}t - f(s)$.
   Applying this to $d(\cdot,h)$ along $G_n$ we have the inequality
$$
d(z,h) > \varphi_1(d(z,x_n))
$$
for every $z\in G_n$ where the slope of $\varphi_1$ is at least $(d_1-d_0-3)/s>0$.

Now since $\seg{x_{-n}x_n}$ lies in a $d_0+1$ neighborhood of $h$, the desired inequality
follows, for an affine function $\varphi$ with the same slope as $\varphi_1$.

The argument for points on $G_{-n}$ is the same, with suitable replacements. 
  \end{proof}

\realfig{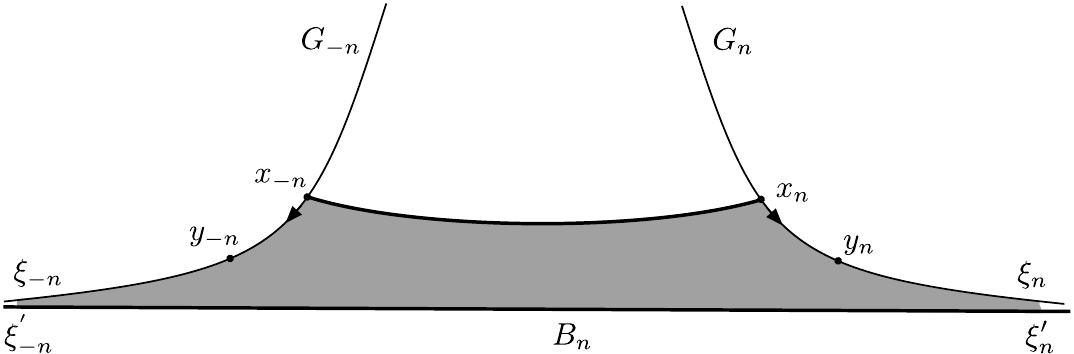}{The ruled hexagon for measuring rate of approach of $B_n$ to
  $G_{\pm n}$}

Now build polygonal loops $P_n,\; n\in \NN$ as follows: 
Choose a point $\xi_n$ on $G_n$ forward of $x_n$ so that $d(\xi_n,B_n) < \delta/2$, and
a point $\xi_{-n}$ on $G_{-n}$ behind $x_{-n}$ so that $d(\xi_{-n},B_n) < \delta/2$. This is
possible because $B_n$ is strongly asymptotic to $G_{-n}$ and $G_n$ in backward and forward times,
and we may choose $\xi_{\pm n}$ as far away from $x_{\pm  n}$ as we like. Let $\xi'_{\pm n}$ denote the nearest points on $B_n$ to
$\xi_{\pm n}$, respectively. Then the loop $P_n$ is the hexagon obtained by connecting the
six points
$$
x_{-n},x_n,\xi_n,\xi'_n, \xi'_{-n}, \xi_{-n}
  $$
in cyclic order using geodesic segments, seen 
in Figure \ref{B-n-hex} as the boundary of the shaded region.
We can triangulate $P_n$ and fill it in with ruled triangles, to obtain a disk $Q_n$ with
negatively curved interior, sides that are geodesic, and six corners at which the exterior
angles (from the point of view of $Q_n$) are at most $\pi$. The Gaussian curvature $\kappa$ in
$Q_n$ is negative, so the Gauss-Bonnet theorem gives us
$$\int_{Q_n} |\kappa| dA \le 4\pi.$$
Now since $G_n$ and $G_{-n}$ are in the $\ep_0$--thick part, we were able to choose
$\delta$ above so that there is an upper bound
$-K_0<0$ for all ambient sectional curvatures at points on a $\delta$--neighborhood of $G_{\pm n}$. This gives a
bound $|\kappa| \ge K_0$ for all points of $Q_n$ that are within distance $\delta$ of the edges
$\seg{x_n\xi_n}$ and $\seg{x_{-n}\xi_{-n}}$.

Now if $\sigma$ is a boundary segment of $Q_n$ on $G_n$ or $G_{-n}$  of length $\lambda$ and $\sigma$ is distance more
than $\delta/2$ from all of the other boundary edges, then it bounds a strip of width $\delta/2$
where $|\kappa| \ge K_0$, and we conclude

\begin{equation}\label{lambda bound}
K_0 \lambda \delta/2 \le \int_{Q_n}|\kappa|dA \le 4\pi
\end{equation}
so that $\lambda \le 8\pi/K_0\delta$. 

Now by Lemma \ref{CAT0 divergence}, for $y\in \seg{x_n\xi_n}$ we have
$$d(y,\seg{x_nx_{-n}}) \ge a d(y,x_n) - c$$
for $a>0$ and $c$ independent of $n$. Let
$\sigma\subseteq \seg{x_n \xi_n}$ be a segment of length at least $8\pi/K_0\delta$ (larger than $\lambda$) starting at distance
$(\delta+c)/a$ from $x_n$. Then every point in $\sigma$ is at least distance $\delta$ from
$\seg{x_{-n}x_n}$. The points on $\sigma$ are also at least distance $\delta$ from $G_{-n}$, for large $n$,  because
$d(x_n,x_{-n})\to \infty$ as $n\to \infty$, which implies that $d(x_n,G_{-n})\to\infty$ using Lemma
\ref{CAT0 divergence} again.
The ``short'' sides of $P_n$ (the boundary of $Q_n$) connecting $\xi_{\pm n}$ to $\xi'_{\pm n}$
may be assumed as far away as we like, so that they are not within distance $\delta$ of $\sigma$. 
Since the inequality (\ref{lambda bound}) is
violated by $Q_n$ as above and $n$ large enough, it follows that there is a point in $\sigma$ which is within distance $\delta/2$ of the
remaining side, which lies on $B_n$.  This is the desired point $y_n$ which is within distance $D:=\frac{(\delta+c)}{a}+\frac{8\pi}{K_0\delta}$ of $x_n$.

We may find $y_{-n}$ using the same argument on $G_{-n}$. 
\end{proof}

\subsection{Choosing $\Phi_n$ to control the length of $\bd Z$}

\begin{lem}\label{Phi constraint for bd Z}
  There exist $D',L>0$ so that for each $n$, if $\Phi_n$ is chosen with axis $A_n$
  sufficiently close to $B_n$, then there are points $v_n,v_{-n}\in A_n$ with
  $$
  \ell_{\bd Z}(v_{\pm n}) \le L
  $$
  and
  $$
  d(v_{\pm n},x_{\pm n}) \le D'
  $$
\end{lem}

\begin{proof}
Let $y_{\pm n}$ be the points obtained in Lemma \ref{rectangle control}. These points are
within $\delta/2$ of $B_n$, so let us choose $\Phi_n$ so that $A_n$ is also within distance $\delta/2$ of
$B_n$ from $\xi_{-n}$ to $\xi_n$ which is again possible by the Density theorem \cite[Theorem 6.1]{bmm1}.
 Let then $v_{\pm n}$ denote points in $A_n$ that are within $\delta$ of $y_{\pm n}$
respectively.

Now $\ell_{\bd Z}(x_{\pm n}) = \ell_{\bd Z}(x)$ because $f_1$ and $ f_2$ fix $\bd Z$. The
segment from $x_n$ to $y_n$  (and $x_{-n}$ to $y_{-n}$) is of length at most $D$ and is in
the $\ep_0$--thick part of Teichm\"uller space (by definition all $G_{\pm n}$ are in an $\ep_0$--thick part of \T space),
 so the lengths of all curves can change only by a bounded factor along such a segment 
 (a thick bounded length WP segment has bounded \T length). 
This gives some uniform bound $L_1$ on $\ell_{\bd Z}(y_{\pm n})$. 

Now by the convexity of the $\delta$--\nbhd of $G_{\pm n}$, the geodesic segment from $y_{\pm n}$ 
to $v_{\pm n}$ also stays in the $\delta$--\nbhd of $G_{\pm n}$ and hence is in the thick part of \T space, and this gives us the desired bound
on $\ell_{\bd Z}(v_{\pm n})$. 
\end{proof}

\subsection{Choosing $\Phi_n$ to control ending laminations}

\begin{lem}\label{Phi constraint for nu}
There exists $A\geq 1$ so that for each $n\in \NN$,  if $\Phi_n$ is chosen with axis $A_n$ sufficiently close to $B_n$, then there is an
  upper bound
  $$ d_Y(\nu^+(\Phi_n),\nu^-(\Phi_n)) \le A $$
for all $Y\subseteq S$ sharing a boundary component with $\bd Z$.
\end{lem}
(Note the statement includes annuli $Y$ with core a component of $\bd Z$).

\begin{proof}
  First we show the bound holds for the laminations $\nu^+(G_n), \nu^-(G_{-n})$.
 Recalling the inequality (\ref{eq:wpbehavior bound}) in the proof of Lemma \ref{lem:dWP-dX} 
 for any curve $\alpha$ in a Bers marking at $x$ we have that
 $d_Y(\alpha,f_1^{-n}(\alpha))\leq D_\alpha$ for all subsurfaces $Y\subseteq S$
 except $X_1$ and annuli with core curves in $\bd X_1$. This implies that for an $A_1\geq 1$ we have
 \[d_Y(x,x_{-n})\leq A_1,\]
 for all subsurfaces $Y$ as above. Similarly, we can see that
 \[d_Y(x,x_{n})\leq A_1\]
 holds for all subsurfaces $Y\subseteq S$
 except $X_2$ and annuli with core curve in $\bd X_2$. Combining the above two inequalities with the triangle inequality
 we see that
   \[d_Y(x_{-n},x_{n})\leq A_2\]
   holds for all subsurfaces $Y\subseteq S$
 except $X_1, X_2$ and annuli with core curves in $\bd X_1$ and $\bd X_2$.

  Moreover, since $G$ is the axis of a \pA map, $\diam_{\calC(Y)}(Q\circ G)$ is uniformly bounded for all $Y\subsetneq S$ (Lemma \ref{lem:dWP-dX}), and it
  follows (applying powers of $f_2$ or $f_1$) that
  \[
  d_Y(x_n,\nu^+(G_n)) \le A_3, \qquad   d_Y(x_{-n},\nu^-(G_{-n})) \le A_3
  \]
 holds for an $A_3\geq 1$ independent of $n$.

  Putting these bounds together we find that
\begin{equation}\label{bound nu G}
  d_Y(\nu^+(G_n),\nu^-(G_{-n})) \le A_4
\end{equation} 
  for all $Y$ except $X_1,X_2$ and annuli with core curves in $\bd X_1$ and $\bd X_2$. 
 Note that this bound holds for all $Y$ sharing a boundary curve with $Z$, since $X_1$ and
 $X_2$ share no boundary curves with $Z$.

Next we recall that $\nu^+(B_n) = \nu^+(G_n)$ and $\nu^-(B_n)=\nu^-(G_{-n})$. Thus
(\ref{bound nu G}) holds for   $(\nu^+(B_n),\nu^-(B_n))$ too. 

Finally, Lemma \ref{lem:dY1} gives us a neighborhood $U^+$ of $\nu^+(B_n)$ in the coarse
Hausdorff topology
such that, for any subsurface $Y$ sharing a boundary component with $Z$,
$$
d_Y(\lambda,\nu^+(B_n)) \le 4
$$
for any $\lambda\in U^+$.
Similarly there is a neighborhood $U^-$ of $\nu^-(B_n)$ with the corresponding property.

Now, the continuity theorem \cite[Theorem 4.7]{bmm1} states that, if $r$ is a recurrent ray
and $r_n \to r$ on compact sets, then the laminations $\nu^+(r_n)$ converge to $\nu^+(r)$
in the coarse Hausdorff topology (The theorem in \cite{bmm1} is stated for a sequence of
rays sharing a basepoint, but the proof applies in general). Thus it follows that, if
$\Phi_n$ is chosen so that $A_n$ is sufficiently close to $B_n$, then $\nu^+(\Phi_n)\in U^+$
and $\nu^-(\Phi_n) \in U^-$. We thus obtain a bound of the form
$$
d_Y(\nu^+(\Phi_n),\nu^-(\Phi_n)) \le A_5
$$
for all $Y$ sharing a boundary component with $Z$ (this includes the annuli with core curves in $\bd Z$).  
This concludes the lemma. 
\end{proof} 

Now we may finish the 

\begin{proof}[Proof of Theorem \ref{thm:closedgeodesic}]
For each $n$, choose $\Phi_n$ so that the conclusions of both Lemma \ref{Phi constraint for bd
  Z}  and Lemma \ref{Phi constraint  for nu} hold. 

Let $v_n,v_{-n}$ be the points on $A_n$ given by Lemma \ref{Phi constraint for bd
  Z}. Then the segments $\seg{v_{-n}v_n}$ satisfy the hypotheses of Theorem \ref{thm
  : bdZ short}: that is, the endpoints are uniformly close to $x_{-n},x_n$ respectively,
and each of those is uniformly close to $p_{-n}$ and $q_{n}$ respectively, which are
points far along the axes $g_1$ and $g_2$, where the geodesic $h$ is close to those
axes. Hence $\seg{v_{-n}v_n}$ is a $D''$-fellow traveller of a long segment $I_n$ in
$h$ so that $\union_n I_n = h$, for some fixed $D''$. Moreover the length of $\bd Z$ is bounded at
the endpoints 
$v_{\pm n}$. Thus Theorem \ref{thm : bdZ short} implies that
$\inf\{\ell_{\bd Z}(x) | x\in {\seg{v_{-n}v_n}}\} \to 0$ as $n\to \infty$.

 Lemma \ref{Phi constraint for nu} gives us the inequality (\ref{eq:dYbd}) for each
 $\Phi_n$. Thus the sequence $\{\Phi_n\}_n$ provides the desired \pA maps to complete the
 proof of Theorem \ref{thm:closedgeodesic}. 
\end{proof}

\section{Comparison with Kleinian surface groups}
\label{sec:kleinian}

In this final section we indicate how Theorems \ref{thm:QF mismatch} and \ref{thm:fibered
  mismatch} can be derived from Theorem \ref{thm:WP mismatch}  and \ref{thm:closedgeodesic}, respectively, using the work of
Brock-Canary-Minsky on Kleinian surface groups \cite{elc1,elc2}. This will show that the
set of short curves along a WP geodesic and the set of short curves in the corresponding
hyperbolic 3-manifold do not necessarily coincide.

Recall first that a {\em Kleinian surface group} is a discrete, faithful representation
$\rho:\pi_1(S) \to PSL(2,\CC)$ which takes punctures of $S$ to parabolic elements (is type preserving). Such a
representation has a pair $(\nu^+,\nu^-)$ of {\em end invariant}, which in particular are
points of $\Teich(S)$ when $\rho$ is quasi-Fuchsian, and are laminations in $\EL(S)$ when
$\rho$ is doubly degenerate.

Let $N_\rho = \HH^3/\rho(\pi_1(S))$ be the quotient hyperbolic $3$--manifold of $\rho$. 
Given $\rho$ and a curve $\gamma$ in $S$ we let $\ell_\gamma(\rho)$ or $\ell_\gamma(N_\rho)$ denote the length of
the geodesic representative of $\rho$ in $N_\rho$. 
The Short Curve Theorem of \cite{elc1} gives the following relationship between small
values of $\ell_\gamma$ and large subsurface projections of the end invariant. 
\begin{thm}\label{thm:minsky}
Suppose that $\rho:\pi_1(S)\to PSL(2,\CC)$ is a Kleinian surface group with end invariant
  $(\nu^+,\nu^-)$ and let $N_\rho=\HH^3/\rho(\pi_1(S))$, then 
\begin{enumerate}
\item for any $A\geq 1$ there is an $\ep>0$ so that if $\ell_\gamma(N_\rho)< \ep$, 
then $\sup_{\substack{Y\subsetneq S: \gamma\subseteq \bd Y}} d_Y(\nu^+,\nu^-)>A$.
\item for any $\ep>0$ there is an $A\geq 1$ so that if $\sup_{\substack{Y\subsetneq S: \gamma\subseteq \bd Y}} d_Y(\nu^+,\nu^-)>A$, 
then $ \ell_\gamma(N_\rho)< \ep$.
\end{enumerate}
\end{thm}

\subsection{Quasi-Fuchsian mismatch}

If $p,q\in\Teich(S)$ we can compare the WP geodesic segment $\seg{pq}$ with the Kleinian
surface group $\rho$ that $QF(p,q)=N_\rho$ has end invariant the pair of Bers markings at $p,q$. 
 We recall the statement of Theorem \ref{thm:QF mismatch}:

\restate{Theorem}{thm:QF mismatch}{
There exists $\ep_1>0$ so that for any $\ep>0$ there is a pair $(p,q)\in
 \Teich(S)\times \Teich(S)$ and a curve $\gamma$ in
 $S$ such that
 $$
 \inf_{x\in \seg{pq}} \ell_\gamma(x) < \ep
 $$
 whereas
 $$
 \ell_\gamma(QF(p,q)) \ge \ep_1.
 $$

}

\begin{proof}
  Given $\ep > 0$, let $p,q$ and $\gamma$ be given by Theorem \ref{thm:WP mismatch}, so
  that 
 $$
 \inf_{x\in \seg{pq}} \ell_\gamma(x) < \ep
 $$
 but
$$
\sup\Big\{d_Y(p,q) \ | \ Y\subseteq S, \gamma\subseteq\bd Y\Big\} \le A.
$$
where $A$ is independent of $\ep$. Then part (1) of Theorem \ref{thm:minsky} gives an
$\ep_1$ such that $\ell_\gamma(QF(p,q)) \ge \ep_1$.  
\end{proof}

\subsection{Fibered mismatch}

If $M_\Phi$ is the mapping torus of a pseudo-Anosov homeomorphism $\Phi\in\Mod(S)$ 
then $M_\Phi$ admits a complete hyperbolic metric by Thurston's geometrization theorem
(see e.g. \cite{Otal:hyperfibered}). The manifold $M_\Phi$ fibers over the circle with fiber $S$, 
and the representation $\rho$ associated to the fiber subgroup has end invariant $(\nu^+,\nu^-)$
equal to the supports of the stable and unstable laminations of $\Phi$.

We are therefore led to compare the short curves of $\rho$ with those of the
Weil-Petersson axis $A_\Phi$ of $\Phi$. We restate Theorem \ref{thm:fibered mismatch} here: 

\restate{Theorem}{thm:fibered mismatch}{
There exists $\ep_1>0$ so that for any $\ep>0$ 
there is a pseudo-Anosov $\Phi\in\Mod(S)$ and a curve $\gamma\in S$
such that
$$\inf_{x\in A_\Phi}\ell_\gamma(x) <\ep$$
whereas
$$
\ell_\gamma(M_\Phi) > \ep_1.
$$

}

\begin{proof}
Given $\ep>0$ let $\Phi$ be the pseudo-Anosov provided by Theorem
\ref{thm:closedgeodesic}, and $\gamma$ the curve such that
\[\inf_{x\in g} \ell_{\gamma}(x) < \ep\]
while
$$
\sup_{\substack{Y\subseteq S: \gamma\subseteq \bd Y}}d_Y(\nu^+,\nu^-)\leq A.
$$
Again by part (1) of Theorem \ref{thm:minsky}, this produces a lower bound
$$
\ell_\gamma(M_\Phi) \ge \ep_1.
$$
\end{proof}

\providecommand{\bysame}{\leavevmode\hbox to3em{\hrulefill}\thinspace}
\providecommand{\MR}{\relax\ifhmode\unskip\space\fi MR }
\providecommand{\MRhref}[2]{%
  \href{http://www.ams.org/mathscinet-getitem?mr=#1}{#2}
}
\providecommand{\href}[2]{#2}


\end{document}